%% file: ms.tex
\documentclass{article}
\usepackage{amssymb, amsfonts, latexsym, amsmath, amsthm, verbatim, hyperref}
\usepackage[pdftex]{color}

\input{rafel-la.def}

\newtheorem{thm}{Theorem}[section]
\newtheorem{cor}[thm]{Corollary}
\newtheorem{lem}[thm]{Lemma}
\newtheorem{fets}[thm]{Facts}
\newtheorem{fet}[thm]{Fact}
\newtheorem{claim}[thm]{Claim}
\newtheorem{prop}[thm]{Proposition}

\newtheorem{con}[thm]{Conjecture}
\theoremstyle{definition}
\newtheorem{defn}[thm]{Definition}
\theoremstyle{definition}

\theoremstyle{remark}
\newtheorem{rem}[thm]{Remark}

\numberwithin{equation}{section}

\numberwithin{equation}{section}

\newcommand{\lcm}{\operatorname{lcm}}
\newcommand{\dpr}{\operatorname{dp-rk}}
\newcommand{\dpra}{\textrm{dp-rank}}
\newcommand{\bdn}{\operatorname{bdn}}
\newcommand{\vc}{\operatorname{vc}}
\newcommand{\mui}{mutually indiscernible}
\newcommand{\muy}{mutual indiscernibility}
\newcommand{\qe}{Quantifier Elimination}
\newcommand{\dimp}{\operatorname{\dim_p}}
\newcommand{\rjn}{\operatorname{RJ_{n}}}
\newcommand{\spn}{\operatorname{Sp_{n}}}
\newcommand{\rjp}{\operatorname{RJ_{p}}}
\newcommand{\rj}{\operatorname{RJ}}
\newcommand{\rjpi}{\operatorname{RJ_{p}^{\infty}}}

\title{Strong  \oag s and dp-rank}
\author{ Rafel Farr\'{e}\\
Universitat Polit\`{e}cnica de Catalunya  \\
Departament de Matem\`{a}tiques  \\
C/Jordi Girona, 1-3 Edifici Omega,\\
  E-08034 Barcelona, \\
Spain
}
\date{\today}%

\begin{document}

\maketitle
\begin{abstract}
We provide an algebraic characterization of strong \oag s: An \oag\ is strong iff it has bounded regular rank and almost finite dimension.
Moreover, we show that any strong \oag\ has finite Dp-rank.

We also provide a formula that computes the exact valued of the Dp-rank of any \oag. In particular characterizing those \oag s with Dp-rank equal to $n$. We also show the Dp-rank coincides with the Vapnik-Chervonenkis density.
\end{abstract}
\newpage
\tableofcontents \vfill

\include{intro}

\include{EQ}

\include{onDP-rank}

\include{Dp-minimalOAG}

\include{ChainConditions}

\include{StrongOAG}

\include{ComputingDpRankInOAG}

\include{VC-density}

\include{GurevichSchmitt}


\bibliographystyle{amsplain}
\bibliography{articles,llibresTM,llibres}

\end{document}

%% file: intro.tex
\section{Introduction}

 S. Shelah  asked the question of which \oag s are strongly dependent (a strong form of NIP). H. Adler  introduced the notion of \emph{strong} theory (a strong form of NTP$_2$) which is a generalization of \emph{strongly dependent} and coincides with the former in NIP theories. In other words, a theory (or more generally a type) is strongly dependent iff it is stong and NIP. Since it is well-known that all \oag s are NIP, the question can be stated as: which \oag s are strong?

One of the main results of the paper is Theorem \ref{thm: caracteritzacio strong}, a characterization of strong \oag s
showing that an \oag\ is strong iff it has bounded regular rank and almost finite dimension. See sections \ref{sec: EQ}   for the definition of bounded regular rank and section \ref{sec: Strong oag s} for the definition of almost finite dimension.

As a corollary we obtain that any strong \oag\ has finite \dpra.
Moreover we provide a formula that computes the exact value of the \dpra.

The paper is organized as follows.

In section \ref{sec: EQ} we state and prove Theorem~\ref{thm: QE}, a \qe\  for \oag s with bounded regular rank. This QE will be applied in sections \ref{sec: dp-minimal ordered groups} and \ref{sec: Strong oag s}. We use a relative QE result of Gurevich and Schmitt to prove our result. The notation and statement of the relative  QE are in section \ref{sec: Gurevich-Schmitt QE for OAG}.

In section \ref{sec: computing dpra} we introduce a variation of ict-patterns by relaxing a little bit the conditions on an ict-pattern. This will be useful to obtain upper bounds of \dpra\ in Theorem \ref{thm: cota superior Dp-rang, cas strong}. 
We call this new patterns wict-patterns (weak independence partition pattern). We will see that wict-patterns compute \dpra. This is the content of Proposition \ref{prop: wict and special patterns compute dp-rank}. The nice point is that when there is QE in the language it is enough to consider wict-patterns constituted by literals (atomic or negation of atomic formulas). This is basically the content of Propositions \ref{prop: no disjunctions} and \ref{prop: no conjunctions}.  Ict-patterns do not have this property since Proposition \ref{prop: no conjunctions} fails for ict-patterns.

We also show in Proposition \ref{prop: single-row} that directed families of formulas cannot occur in two different rows of the same pattern. this is also very useful in order to provide upper bounds on the \dpra.

In section \ref{sec: dp-minimal ordered groups} we provide a characterization of dp-minimal \oag s. The arguments of this section are the same as those in sections \ref{sec: Strong oag s}. Moreover the results here are a corollary as those in section \ref{sec: Strong oag s}. This section is just a warming-up of  section \ref{sec: Strong oag s}.

The algebraic characterization of a dp-minimal \oag\ as having finite dimension has been first published in \cite{jahnke_simon_walsberg_2017} using another terminology.

In section \ref{sec: Chain conditions} we prove Propositions \ref{prop: cota inferior burden en grups, cas finit} and \ref{prop: cota inferior burden en grups, cas numerable}. They may be considered as chain conditions Theorems, see \cite{Kaplan-Shelah:Chain2013} for close results. They are useful in proving lower bounds for burden (and \dpra). In particular \ref{prop: cota inferior burden en grups, cas numerable} is useful to obtain consequences from strongness. It is used in \ref{prop: infinite p-dim=inf implica no strong} and \ref{thm: p-regular rank infinit implica no strong}. Propositions \ref{prop: cota inferior burden en grups, cas finit} is used in section \ref{sec: computing dpra in oag}.

Section \ref{sec: Strong oag s} contains Theorem \ref{thm: caracteritzacio strong}, the characterization of strong \oag s. It also contains the proof that all strong \oag s have finite \dpra.

Section \ref{sec: computing dpra in oag} is a refinement of the arguments of section 6 to obtain a formula in Theorem~\ref{thm: formula Dp-rang, cas RR acotat} that computes the exact value of the \dpra.

The characterisation of Strong \oag s as having bounded regular rank and almost finite dimension is obtained independently in \cite{Hal-Has2017:StronglyDependentOAGandHenselianFields} using another terminology and different proof techniques. The authors also prove in \cite{Hal-Has2017:StronglyDependentOAGandHenselianFields}  that any strong \oag\ has finite \dpra. They
 also provide a formula similar to \ref{thm: formula Dp-rang, cas RR acotat}. Corollary \ref{cor: formula Dp-rank d'una suma directa de OAGs} is also there.

 Section~\ref{sec: VC-density} contains a proof that for \oag s, the VC-density coincides with the \dpra. More precisely Theorem~\ref{thm: Dp-rank = Vc-densitat} shows that the \dpra\ coincides with the VC-density function evaluated at $1$.

Most of the results of this paper (sections \ref{sec: EQ} , \ref{sec: computing dpra}, \ref{sec: dp-minimal ordered groups} ,\ref{sec: Chain conditions}, \ref{sec: Strong oag s} and \ref{sec: computing dpra in oag} )were obtained in 2014, but they remained unpublished until now. They were  exposed in detail in the Barcelona Model theory seminar along various sessions during March and April 2014 (see \url{http://www.ub.edu/modeltheory/mt.htm}). I also gave a talk about the subject in the Mathematical Logic Seminar in Freiburg the 25th. of June 2014 (see
\url{http://logik.mathematik.uni-freiburg.de/lehre/archiv/ss14/oberseminar-ss14.html
}).
\subsection{notation}

We use  $\ov x$ to denote a tuple of variables. Most of the time is finite but sometimes could be infinite. 
We denote by $\abs{\ov x}$ the length of the tuple. We use $x$ to denote a single variable. The same conventional notation will be used for tuples of elements.

We will denote by $\PP$ the set of all prime numbers.

We always play with the monster model of some complete theory $T$. But in fact a little bit of saturation (a weakly saturated model) is enough.

The notation an terminology specific for \oag s is contained in sections \ref{sec: EQ} and \ref{sec: Gurevich-Schmitt QE for OAG}.

When we say that a structure  is VCA-minimal, dp-minimal, strong (or in general any abstract property of theories) we mean the theory of the structure is VCA-minimal, dp-minimal, strong... In practice we always may assume the structure is as much saturated as needed (by replacing it by another elementarily equivalent). The same applies to invariants of a theory, like $\dpr(T)$, $\bdn(T)$... In this case we denote  $\dpr(M)=\dpr(\operatorname{Th}(M))$ and similarly by $\bdn$ and VC-density.

When dealing with  \dpra\ and burden there two possible ways to define them, both common in the literature. 
The first possibility is to define the \dpra\ of a type as the supremum of all depths of all ict-patterns for the type. The second one is Shelah's style: the the \dpra\ of a type is the first cardinal $\ka$ for which there is no ict-patterns for the type of depth $\ka$ (and $\infty$ if such cardinal does not exists).
We choose the first definition, although we do not use subscripts to distinguish whether those supremums are achieved or not.

If $\vf(\ov x,\ov y)$ is a formula, $M$ is a model and $\ov b$ is a tuple in $M$ of the same length as $\ov y$, we will use $\vf(M,\ov b)$ to denote the set defined by $\vf(\ov x,\ov b)$ in $M$, namely $\set{\ov a\in M^{\abs{\ov x}} \mid M\models \vf(\ov a,\ov b)}.$

If $G$ is an \oag\ we will denote $\De\unlhd G$ to indicate that $\De$ is a convex subgroup of $G$ and $\De\lhd G$ that $\De$ is a proper convex subgroup.

%% file: EQ.tex
\section{A Quantifier Elimination  for \oag s with bounded regular rank}\label{sec: EQ}

In this section we generalize a quantifier elimination result of W. Weispfenning in \cite{Wei81}. Weispfenning result is precisely our theorem \ref{thm: QE} in the case when $G$ has finite regular rank, see the comments after proposition \ref{prop: caracteritz. de regular rank bounded}. We deduce our result from a more general QE of Y. Gurevich and P. Schmitt, see \cite{Gu65}, \cite{Schm82a}, \cite{Schm82b}. This more general result, the notation and the terminology is explained in section \ref{sec: Gurevich-Schmitt QE for OAG} of this paper. If $p$ is a prime number, $G/pG$ is a vector space over the field with $p$ elements. We call the dimension of this vector space the \textbf{$p$-dimension} of $G$ and we denote it by $\dimp(G)$.

We begin with a definition.
\begin{defn}
Let $n\ge 2$. The $n$-regular rank of $G$ is the order-type of the following ordered set:
$$
\left(\rjn(G), \se \right),
$$
where $\rjn(G)$ denotes $\set{A_{n}(g) \mid g\in G-\set{0}}$, the set of all $n$-regular jumps.
\end{defn}
The $n$-regular rank of an \oag\ $G$ is a linear order. When it is finite we can identify this order with its cardinal. Moreover it can be easily characterized without any reference to the sets $A_n(g)$ as next remark shows.
\begin{rem}\label{rem: rang n-regular finit}
\begin{enumerate}
  \item $G$ has $n$-regular rank $0$ iff $G=\set{0}$.
  \item  $G$ has $n$-regular rank $1$ iff $G$ is $n$-regular and non-trivial.
  \item  $G$ has $n$-regular rank equal to  $m$ iff there are $\tiraa {\De}0m$ convex subgroups of $G$, such that:
      \begin{enumerate}
       \item $\set{0}=\De_0\lhd\De_1\lhd \cdots \lhd\De_m=G$
       \item $\De_{i+1}/\De_i$ is $n$-regular for $0\le i<m$
       \item $\De_{i+1}/\De_i$ is not  $n$-divisible for $0< i<m$.
      \end{enumerate}
      In this case $\rjn(G)=\set{\tiraa{\De}0{{m-1}}}$.
  \item It $G$ has finite $n$-regular rank and $H\equiv G$ then $H$ has the same $n$-regular rank as $G$. Hence, when finite, the $n$-regular rank is an invariant of the theory of $G$. If $G$ has infinite $n$-regular rank, the regular rank of a $\ka$-saturated model of the theory of $G$ is a linear order which has cardinality at least $\ka$. It could be interesting characterize the $n$-regular rank of the monster model (it is not an $\eta_\al$-ordered set in general).
  \item If $\dim_p(G)$ is finite then the $p$-regular rank of $G$  is at most $\dimp(G)+1$.
  \item The number of convex subgroups $\De$ of $G$ with $G/\De$ discrete is bounded by the $n$-regular rank for each $n$. In particular, if $G$ has finite $n$-regular rank for some $n$, then the number of convex subgroups $\De$ of $G$ with $G/\De$ discrete is finite.
  \item If the $n$-regular rank of $G$ is finite then each convex subgroup of the form $F_{n}(x)$ belongs to $\rjn(G)$.
\end{enumerate}
\end{rem}
\begin{proof}
\emph{1} and \emph{2} are particular cases of \emph{3},  and \emph{5} follows from \emph{3}.

\emph{3}. If we have such a chain then $\rjn(G)=\set{\tiraa{\De}0{{m-1}}}$, since $g\in\De_{i+1}-\De_i$ implies $A_{n}(g)=\De_i$.

\emph{4}. $G$ has $n$-regular rank m iff in $G$ holds:
$$
\exists \tiraa y0{{m-1}}\forall z\left(\bigvee_{j=0}^{m-1} A_n(z)=A_n(y_j)\land  \bigwedge_{j=0}^{m-2}A_n(y_j)\subset A_n(y_{j+1})  \right)
$$

\emph{6}. Because $G/\De$ discrete with \fpe\ $1_\De$ implies  $\De=A_n(1_\De)$ and thus $\De\in \rjn(G)$ for any $n\ge2$.

\emph{7}.  Assume $\rjn(G)=\set{ \tiraa {\De}0{{m-1}} }$ as in \emph{3} and  $\De_i\lhd F_{n}(x) \lhd \De_{i+1}$.  Then $x\in \De_{i+1}+nG$. By regularity of the jump, $\De_{i+1}/F_{n}(x)$ is $n$-divisible, thus $\De_{i+1}=F_{n}(x)+n\De_{i+1}$. Then $x\in F_{n}(x)+nG$, a contradiction.
\end{proof}
We also denote by $\rj(G)$ the set $\bigcup_{n\ge 2}\rjn(G)$ an call it the set of regular jumps of $G$.

\begin{prop}\label{prop: caracteritz. de regular rank bounded}
Let $G$ be an \oag. The following are equivalent:
\begin{enumerate}
  \item  $G$ has finite $p$-regular rank for each prime $p$.
  \item  $G$ has finite $n$-regular rank for each $n\ge 2$.
   \item There is some cardinal $\kappa$ such that for any $H\equiv G$, $\abs{\rj(H)}\le \ka$ ($\rj(H)$ is countable).
   \item For any $H\equiv G$, any definable convex subgroup of $H$ has a definition without parameters.
  \item  There is some cardinal $\kappa$ such that for any $H\equiv G$, $H$  has at most $\ka$  (countably many) definable  convex subgroups.
\end{enumerate}
Moreover, in this case, $\rj(G)$ is the collection of all proper definable convex subgroups of $G$ and all are definable without parameters.
\end{prop}
\begin{proof}
\emph{1}$\Rightarrow$\emph{2}.
$\rjn(G)\se\bigcup_{p\mid n}\rjp(G)$ because $A_{n}(g)=\bigcup_{p\mid n}A_p(g)$.

\emph{2}$\Rightarrow$\emph{3}. Let $H\equiv G$.  By Point 4 in Remark \ref{rem: rang n-regular finit} $H$ has also finite $n$-regular rank for each $n$ so $\rj(H)$ is countable.

\emph{3}$\Rightarrow$\emph{4} Let $H\equiv G$. A compactness argument shows that $H$ must have finite $n$-regular rank for each $n$. By Theorem 4.1 of  \cite{De-Fa96} any proper definable convex subgroup of $H$ is an intersection of elements of $\rjn(H)$ for some $n$. The finiteness of $\rjn(H)$ implies that $\rj(H)$ is the set of all proper definable convex subgroups of $H$. Now, if $\rjn(H)=\set{\tiraa{\Delta}0{{m-1}}}$ and $\De_i=A_n(g_i)$ then $\De_i$ can be defined in $H$ by the formula
$$
\exists \tiraa y0{{m-1}}\left(  x\in A_n(y_i)\land  \bigwedge_{j=0}^{m-2}A_n(y_j)\subset A_n(y_{j+1})  \right)
$$

\emph{4}$\Rightarrow$\emph{5} is obvious.

\emph{5}$\Rightarrow$\emph{1} If for some $p$, $\rjp(G)$ is infinite, a compactness argument allows us to find $H\equiv G$ where $\rjp(H)$ is as big as we want.

Finally observe that $\rjn(G)=\bigcup_{p\in\PP}\rjp(G)$.
\end{proof}

In view of point \emph{3} of Proposition \ref{prop: caracteritz. de regular rank bounded} will say that an \oag\ has \textbf{bounded regular rank } if it satisfies any of the conditions of Proposition \ref{prop: caracteritz. de regular rank bounded}. By Facts~\ref{fets: A}.\emph{7}, $\rjn=\spn$ when $\rjn$ is finite. Therefore an \oag\ has bounded regular rank iff all the spines $\spn(G)$ are finite.
In this case we define the \textbf{regular rank} of $G$ as $\abs{\rj(G)}$, the cardinality of $\rj(G)$. It is either finite or $\aleph_0$. Observe that the property of having finite or bounded regular rank depends only in the theory of $G$. Also the value of the regular rank  depends only on the theory (we may say that the regular rank is $\infty$ or non-defined when the group has not bounded regular rank). 


One can easily check that Weispfenning QE  in \cite{Wei81}  is just Theorem \ref{thm: QE} in the case of finite regular rank.
\begin{thm}\label{thm: QE}
Let $G$ be an \oag\ with bounded regular rank. Then $G$ admits QE in the following language:
\begin{align}\label{eq: languageQE}
L=&    \set{+,-,0,\le} \ \cup \   \set{1_\De \mid \De\in\rj(G),\ G/\De \text{\textrm{ discrete}} } \ \cup    \nonumber \\
&    \set{x\equiv y \mod \De \mid \De\in\rj(G)}  \ \cup \          \nonumber                \\
&     \set{x\equiv y \mod(\De+p^mG) \mid p\in\PP,\ \De\in\rjp(G),\ m\ge 1}    \nonumber
  \end{align}
Here, if $G/\De$ is discrete, $1_\De$ denotes an element of $G$ whose projection to $G/\De$ is the smallest positive element.
\end{thm}

\begin{proof}
We begin by remarking that we can add the following predicates for free:
\begin{equation}\label{eq: widelanguageQER}
      \set{x\equiv y \mod(\De+nG) \mid \De\in\rj(G), \ n\ge 1}  \nonumber
\end{equation}
In fact we could have added any set of predicates of the form $$x\equiv y \mod(\De+nG)$$ even if $\De$ is not definable. This is because:
\begin{claim}
\begin{enumerate}
  \item If $n=p_1^{r_1}\cdots p_k^{r_k}$ then
\begin{equation}\nonumber
      x\equiv y \mod(\De+nG) \Longleftrightarrow  \bigwedge_{i}  x\equiv y \mod(\De+p_i^{r_i}G)
\end{equation}
    \item If $\De\notin\rjp(G)\cup\set{G}$, let $\De_1,\De_2$ be consecutive elements of $\rjp(G)\cup\set{G}$ such that $\De_1\subset \De\subset\De_2$. Then
\begin{equation}\nonumber
      x\equiv y \mod(\De+p^rG) \Longleftrightarrow  x\equiv y \mod(\De_2+p^rG)
\end{equation}
\end{enumerate}
\end{claim}

\noindent\emph{Proof of the claim:}
\emph{1} follows from the following formula:
$$
\ds \bigcap_{i=1}^t\De+m_iG=\De+\lcm(\tira mt)G.
$$
 To prove this
it is enough to prove $(\De+nG)\cap (\De+mG)=\De+\lcm(n,m)G$. The other inclusion being obvious, it suffices to prove $\subseteq$. Assume $a=\de_1+mg_1=\de_2+ng_2$ with $\de_i\in \De$ and $g_i\in G$. Denote $d=\gcd(m,n)$, $m=dm'$, $n=dn'$. Let $1=\la m'+\mu n'$ be a B\'{e}zout identity for $n',m'$. Then $a=\la m'a+\mu n'a=\la m'(\de_2+ng_2)+\mu n'(\de_1+mg_1)=(\la m'\de_2+\mu n'\de_1)+\lcm(n,m)(\la g_2+\mu g_1)\in \De+\lcm(n,m)G.$

\emph{2} follows because $\De_2/\De$ $p$-divisible implies $\De_2= \De+p^r\De_2$ and thus $\De+p^rG=\De_2+p^rG $.
\emph{This ends the proof of the claim.}

We will use theorem \ref{thm: B}. Let $n$ be as in the statement of  \ref{thm: B}. Since $\rjn(G)$ is the domain of $\operatorname{Sp_{n}}(G)$ and is finite ,
$\operatorname{Sp_{n}}(G)\models\psi_0(C_1,\dots,C_m,D_1,\dots,D_r)$ is equivalent to a boolean combination of the following formulas
$A_{n}(t_i(\ov g))=\De$ and $F_{n}(s_i(\ov g))=\De$ with $\De\in\rjn(G)\se \bigcup_{p\mid n}\rjp(G)$. The following claim shows that this has a quantifier-free translation into the language $L$:
\begin{claim}
\begin{enumerate}
  \item
$$  
      A_{n}(x)=\De \Longleftrightarrow  (x\not \equiv 0 \mod\De) \wedge   (x\equiv 0 \mod\De^{n+}  )
$$
    \item
$$
       F_{n}(x)=\De \Longleftrightarrow  (x\not \equiv 0 \mod(\De+nG)) \wedge  ( x \equiv 0 \mod(\De^{n+}+nG)),
$$
\end{enumerate}
where $\De^{n+}$ denotes the successor of $\De$ in $\rjn(G)$.
\end{claim}
It remains to prove that the LOG$^*$-predicates  $M_k$, $E_{(n,k)}$ and $D_{(p,r,i)}$ are expressible without quantifiers in $L$.
This is is done in the next claim. This ends the proof of the theorem.
\begin{claim}
\begin{enumerate}
  \item $\ds M_k(x)  \Longleftrightarrow \bigvee_{G/\De\text{ discrete}}   \!\!\!  x\equiv k 1_\De\mod \De$
  \item $\ds E_{(n,k)}(x)  \Longleftrightarrow   \!\!\!\!\!\!\!\!\!\!\!  \bigvee_{G/\De\text{ discrete}}  \!\!\!\!\!\!\!\!\!\! \left( x\equiv 0\mod (\De^{n+}+nG) \right)   \land  \left(x\equiv k 1_\De\mod (\De+nG) \right)  $
  \item $\ds D_{(p,r,i)}(x)     \Longleftrightarrow (x\equiv 0\mod p^rG )\lor   \!\!\!\!  \bigvee_{G/\De\text{ discrete}}    \!\!\! \Bigl(  \left(x\equiv 0\mod (\De+p^iG)\right) \ds \land  \left(x\equiv 0\mod (\De^{p+}+p^rG) \right)\land  \left(x\not\equiv 0\mod (\De+p^rG) \right) \Bigr)$
\end{enumerate}
\end{claim}

\noindent\emph{Proof of the claim:}
\emph{1} is easy, lets see \emph{2.}  By remark \emph{7} in \ref{rem: rang n-regular finit}, assume $F_{n}(x)=\De$ for some $\De\in\rjn(G)$.  It is easy to see that in this case $\Ga_{1,n}(x)=\De+nG$ and $\Ga_{2,n}(x)=\De^{n+}+nG$. Moreover, if $G/\De$ is discrete,  $\cla{x}=k\cla{1_\De}$ in $\Ga_{n}(x)$ iff $x\equiv k1_\De \mod\De+nG$.
Now, keeping in mind that $F_{n}(x)=\De$ is equivalent to $(x\equiv 0\mod \De^{n+}+nG)\land(x\not\equiv 0\mod \De+nG)$, we get
the result. One must bear in mind also that, since $k>0$,  $x\equiv k1_\De \mod\De+nG$ implies $(x\not\equiv 0\mod \De+nG)$.

To prove \emph{3}, assume $F_{p^r}(x)=\De$ for some $\De\in\operatorname{RJ_{p^r}}(G)$. As before, $\Ga_{1,p^r}(x)=\De+p^rG$ and $\Ga_{2,p^r}(x)=\De^{p^r+}p^rG$. Then $\cla{x}\in p^i\Ga_{p^r}(x)$ iff $x\in \De+p^i\De^{p^r+}+p^rG=(\De+p^iG)\cap(\De^{p^r+}+p^rG)$. Moreover, $\De^{p^r+}=\De^{p+}$, since for \oag s, being $p$-regular ($p$-divisible) is equivalent to being $p^r$-regular ($p^r$-divisible).
\emph{This ends the proof of the claim.}
\end{proof}

%% file: onDP-rank.tex
\section{On computing \dpra}\label{sec: computing dpra}

Along this section we work again with the monster model of some complete theory $T$.
We begin by recalling the definition of ict-pattern and \dpra.

The following was   defined in \cite{Usvyatsov-genericallyStable} and \cite{onshuus_Usvyatsov-Dp-min&strong}. 
\begin{defn}
Let $p(\ov x)$ be a partial type. An ict-pattern for $p(\ov x)$ consist in the following data: a sequence of formulas $S:=(\vf(\ov x,\ov y_i) \mid i\in k)$ and an array $A:=(\ov a_j^i \mid i\in\ka,\ j\in O)$ of parameters, where $\ka$ is a cardinal number and $O$ is an infinite linearly ordered set \st:

for every $f\in O^\ka$, the following set of formulas is consistent with $p(\ov x)$:
$$
\Ga_f^S(A):=\set{\vf_i(\ov x,\ov a_{f(i)}^i) \mid i\in \ka}\cup\set{\lnot\vf_i(\ov x,\ov a_j^i) \mid i\in \ka, j\ne f(i)}
$$
The cardinal number $\ka$ is called the depth of the pattern and we allow $\ka$ to be finite.  We also say that it is an ict-pattern of type $\ka\times O$ with sequence of formulas $S$ and array $A$.

The \dpra\ of $p(\ov x)$, denoted by $\dpr(p(\ov x))$ is the supremum of all depths of all ict-patterns for $p(\ov x)$.

If $T$ denotes a complete theory with a main sort \footnote{there is a sort, called main sort, such that all other sorts are obtained as sorts of imaginaries of the theory of the main sort alone} the \dpra\ of $T$, denoted by  $\dpr(T)$ is  $\dpr(x=x)$ where $x$ is a single variable for the main sort.

\end{defn}

Now we introduce wict-patterns.
\begin{defn}
Let $p(\ov x)$ be a partial type. A \textbf{wict-pattern} for $p(\ov x)$ consist in the following data: a sequence of formulas $S:=(\vf(\ov x,\ov y_i) \mid i\in k)$ and an array $A:=(\ov a_j^i \mid i\in\ka,\ j\in O)$, where $\ka$ is a cardinal number and $O$ is an infinite linearly ordered set \st:

for every $f\in O^\ka$, the following set of formulas is consistent with $p(\ov x)$:
$$
\De_f^S(A):=\set{\vf_i(\ov x,\ov a_{f(i)}^i) \mid i\in \ka}\cup\set{\lnot\vf_i(\ov x,\ov a_j^i) \mid i\in \ka, j>f(i)}
$$
The cardinal number $\ka$ is called the depth of the pattern and we allow $\ka$ to be finite.  We also say that it is a wict-pattern of type $\ka\times O$ with sequence of formulas $S$ and array $A$.
\end{defn}

\begin{rem}\label{rem: observ. sobre wict-patterns}
\begin{enumerate}
  \item By the same arguments as for ict-patterns, we can replace the parameters of a wict-pattern (without changing the sequence of formulas nor the depth) by another array with \mui\ rows over the set of parameters of the type. This means that each row is indiscernible over the set of parameters of the type plus all other rows.  
  \item When the wict-pattern has \mui\ rows over the set of parameters of $p$, it is enough to check the consistency of a single path: If $p(\ov x)\cup\De_f^S(A)(\ov x)$ is consistent for some $f\in O^\ka$ then the same holds for any $f\in O^\ka$.
  \item The previous fact is not true in general for ict-patterns, it depends on the order-type of $O$. For instance, it holds  for ict-patterns of kind $\ka\times\Z$, but not for ict-patters of kind $\ka\times \om$ (as the formula $x>y$ shows in the theory on dense linear order without endpoints).
\end{enumerate}
\end{rem}

The following kind of patterns were also first used by Shelah. The name `special' is not standard and is simply used to distinguish them from the other two.

\begin{defn}
Let $p(\ov x)$ be a partial type. A special pattern for $p(\ov x)$ consist in the following data: a sequence of formulas $S:=(\vf(\ov x,\ov y_i) \mid i\in k)$ and an array $A:=(\ov a_j^i \mid i\in\ka,\ j\in \om)$ with \mui\ rows, where $\ka$ is a cardinal number \st\ the following set of formulas is consistent:
\begin{equation}\label{eq: special pattern}
p(\ov x)\cup\set{\vf_i(\ov x,\ov a_0^i) \mid i\in \ka}\cup\set{\lnot\vf_i(\ov x,\ov a_1^i) \mid i\in \ka}
\end{equation}
\end{defn}

The following shows we can replace ict-patterns by wict-patterns or special patterns  in order to compute \dpra.
\begin{prop}\label{prop: wict and special patterns compute dp-rank}
Let $p(\ov x)$ be a partial type and $\ka$ a cardinal number. The following are equivalent:
\begin{enumerate}
  \item There is an ict-pattern for $p$ of depth $\ka$.
  \item There is a wict-pattern for $p$ of depth $\ka$.
  \item There is an special pattern for $p$ of depth $\ka$.
\end{enumerate}
\end{prop}
\begin{proof}
\emph{\emph{1}$\Rar$\emph{2}} is obvious.

\emph{\emph2}\emph{$\Rar$3}. By a standard argument we may replace the array by a $\ka\times\om$-array with \mui\ rows.

\emph{\emph{3}$\Rar$\emph{1}}. 
A compactness argument allows us to replace the array by a \mui\  array of type $\ka\times\Z$.
Now we fix some witness $\ov b$ of the consistency of \eqref{eq: special pattern}. Moreover we fix a row $i\in\ka$. By deleting some positions we may assume the truth-value of $\vf_i(\ov b,\ov a^i_j)$ for $j<0$ is constant. We may also assume the truth-value of $\vf_i(\ov b,\ov a^i_j)$ for $j>1$ is constant. By replacing the formula $\vf_i(\ov x,\ov a^i_j)$ by $\psi_i(\ov  x,\ov a^i_j\ov a^i_{j+1}):=\vf_i(\ov x,\ov a^i_j)\leftrightarrow\lnot\vf_i(\ov x,\ov a^i_{j+1})$ one achieves a new pattern with negative values on the left (with a possible exception), negative values on the right (with a possible exception) and a positive value in the 0-th. column. More precisely the formula $\psi_i(\ov x,\ov u_i)$ is $\vf_i(\ov x,\ov y_i)\leftrightarrow\lnot\vf_i(\ov x,\ov z_i)$ and the tuple is $\ov c_i:=\ov a^i_{2j}\ov a^i_{2j+1}$. By deleting at most two positions we obtain a row where $\psi(\ov b,\ov c^i_j)$ is false for $j\ne 0$ and true for $j=0$. As the new array is also \mui, by point 3 of Remark \ref{rem: observ. sobre wict-patterns} we have obtained an ict-pattern of type $\ka\times\Z$ for $p$.
\end{proof}


The following lemma allows us to avoid disjunctions in wict-patterns. It is well-known for ict patterns.
\begin{prop}\label{prop: no disjunctions}
 Assume there is a wict-pattern( for $p$ with formulas $S=(\vf_i(\ov x,\ov y_i)\mid i\in\ka)$ and for each $i\in\ka$ $\vf_i(\ov x,\ov y_i)=\bigvee_{k=0}^{n_i}\psi^k_i(\ov x,\ov y_i)$. Then for some $f\in \om^\ka$ with $f(i)\le n_i$ there is a  wict-pattern for $p(x)$ with formulas $(\psi^{f(i)}_i(\ov x,\ov y_i)\mid i\in\ka)$. The same is true for ict-patterns.
\end{prop}
\begin{proof}
Let $(\ov a_j^i)_{i\in\ka,j\in \om}$ be a \mui\ $\ka\times\om$ array showing there is a wict-pattern for $p$ with formulas $S$. There is some $\ov b$ such that $\models\vf_i(\ov b,\ov a_0^i)$ and $\not\models\vf_i(\ov b,\ov a_j^i)$ for all $i\in \ka$ and $j>0$. Hence, for some $f\in \om^\ka$ with $f(i)\le n_i$ $\models\psi^{f(i)}_i(\ov b,\ov a_0^i)$ for all $i\in \ka$. By point 2 in Remark \ref{rem: observ. sobre wict-patterns},   $(\ov a_j^i)_{i\in\ka\\ j\in \om}$ together with $(\psi^{f(i)}_i(\ov x,\ov y_i)\mid i\in\ka)$ is a $\ka\times\om$ wict-pattern for $p$. The argument for ict-patterns is the same starting with an array of type $\ka\times\Z$.
\end{proof}

Now, we see that for wict patterns, the same holds for conjunctions.
\begin{prop}\label{prop: no conjunctions}
 Assume there is a wict-pattern for $p(\ov x)$ with sequence of formulas $S=(\vf_i(\ov x,\ov y_i)\mid i\in\ka)$ and $\vf_i(\ov x,\ov y_i)=\bigwedge_{k=0}^{n_i}\psi^k_i(\ov x,\ov y_i)$ for each $i\in\ka$.  Then for some $f\in \om^\ka$ with $f(i)\le n_i$ there is a wict-pattern for $p(\ov x)$ with formulas $(\psi^{f(i)}_i(\ov x,\ov y_i)\mid i\in\ka)$.
\end{prop}
\begin{proof}
If $(\ov a_j^i)_{i\in\ka, j\in \om}$ is a \mui\  $\ka\times\om$ wict-pattern with sequence $S$ there is some $\ov b$ such that $\models\vf_i(\ov b,\ov a_0^i)$ and $\not\models\vf_i(\ov b,\ov a_j^i)$ for all $i\in \ka$ and $j\ne0$.
Fix some row $i$.  Since  $\vf_i(\ov b,\ov a_j^i)$  is false for $j>0$, there is some $k\in\set{1,\ldots n_i}$ such that $\psi^k_i(\ov b,\ov a^i_j)$ fails for infinitely many $j>0$. By deleting elements in the row, we may assume  that $\psi^k_i(\ov b,\ov a^i_j)$ fails for all $j>0$. Hence $\models\psi^k_i(\ov b,\ov a_0^i)$ and $\not\models\psi^k_i(\ov b,\ov a_j^i)$ for all and $j\ne0$. Since this may be done for any $i$, we get $f\in \om^\ka$ such that $\models\psi^{f(i)}_i(\ov b,\ov a_0^i)$ and $\not\models\psi^{f(i)}_i(\ov b,\ov a_j^i)$ for all $i\in \ka$ and $j\ne0$.
By By point 2 in Remark \ref{rem: observ. sobre wict-patterns}, $(\ov a_j^i)_{i\in\ka,\ j\in \om}$ together with $(\psi^{f(i)}_i(\ov x,\ov y_i)\mid i\in\ka)$ forms a $\ka\times\om$ wict-pattern for  $p$.
\end{proof}
Proposition \ref{prop: no conjunctions} does not hold for ict-patterns as the following example shows. In the theory of dense linearly ordered sets without endpoints there is an ict-pattern of depth $1$ with formula $y_1<x<y_2$ but there is no ict-pattern for none of the formulas $y_1<x$ nor $x<y_2$.

However we can reduce any conjunction to a conjunctions of  at most two formulas in ict-patterns:
\begin{prop}
 Assume there is an ict-pattern for $p(\ov x)$ with sequence of formulas $S=(\vf_i(\ov x,\ov y_i)\mid i\in\ka)$ and $\vf_i(\ov x,\ov y_i)=\bigwedge_{k=0}^{n_i}\psi^k_i(\ov x,\ov y_i)$ for each $i\in\ka$.  Then for every $i\in \ka$ there is some $I_i\se\set{0,\ldots,n_i}$ with at most two elements \st\ there is an  ict-pattern for $p(x)$ with formulas  $(\bigwedge_{j\in I_i}\psi^j_i(\ov x,\ov y_i)\mid i\in\ka)$.
\end{prop}
\begin{proof}
If  $S$ together $(\ov a_j^i)_{i\in\ka, j\in \Z}$ is a \mui\  $\ka\times\Z$ ict-pattern for  $p$, there is some $\ov b$ such that $\models\vf_i(\ov b,\ov a_0^i)$ and $\not\models\vf_i(\ov b,\ov a_j^i)$ for all $i\in \ka$ and $j\ne0$.

Fix some row $i$.  Since  $\vf_i(\ov x,\ov a_j^i)$  is false for $j>0$, there is some $k\in\set{1,\ldots n_i}$ such that $\psi^k_i(\ov b,\ov a^i_j)$ fails for infinitely many $j>0$. By deleting elements in the row, we may assume  that $\psi^k_i(\ov b,\ov a^i_j)$ fails for all $j>0$. The same happens for for $j<0$ with maybe another formula. These two formulas provide the set $I_i$.
Hence,  we get that  $\models\bigwedge_{j\in I_i}\psi^j_i(\ov b,\ov a_0^i)$ and $\not\models\bigwedge_{j\in I_i}\psi^j_i(\ov b,\ov a_j^i)$ for all $i\in \ka$ and $j\ne0$.
By point 3 in Remark \ref{rem: observ. sobre wict-patterns}, $(\ov a_j^i)_{i\in\ka, j\in \Z}$ is the  $\ka\times\Z$ array of an ict-pattern with formulas $(\bigwedge_{j\in I_i}\psi^j_i(\ov x,\ov y_i)\mid i\in\ka)$.
\end{proof}

The following definition is standard, see \cite{Adl:NIP} or \cite{ADHMS:2} We include it here for a better readability of the paper.

\begin{defn}
  Let  $\vf(\ov x,\ov y)$  be a partitioned formula. The dual-alternation number of  $\vf(\ov x,\ov y)$ is  the maximum number of changes of the truth-value of  $\vf(\ov b,\ov a_i)$ (when $i$ increases in $O$) for any tuple $\ov b$ and any indiscernible sequence $(\ov a_i,i\in O)$.
\end{defn}

\begin{rem}
A (partitioned) formula $\vf(\ov x,\ov y)$ has dual-alternation number equal to zero iff  $\set{\vf(M,\ov a) \mid \ov a\in M}$ is a finite set in any (some) model. We consider that a partitioned formula $\vf(\ov x;\ov y)$ with $\ov y$ the empty tuple has dual-alternation number zero.
\end{rem}
\begin{proof}
If the set $\set{\vf(M,\ov a) \mid \ov a\in M}$ is finite, in an indiscernible sequence $(\ov a_i\mid i\in\om)$ some repetition must occur: $\vf(M,\ov a_i)=\vf(M,\ov a_j)$ for some $i\ne j$ hence $\vf(M,\ov a_i)$ should be constant. This implies that the dual-alternation number of $\vf(\ov x,\ov y)$ is zero. Conversely, if the set $\set{\vf(M,\ov a) \mid \ov a\in M}$ is infinite, by a standard compactness plus Ramsey argument (see for instance Lemma 5.1.3 of \cite{TentK-ZieglerM:CMT}),
one can construct an indiscernible sequence $(\ov a_i\mid i\in\om)$ where $\vf(M,\ov a_i)$ is not constant. Hence the dual-alternation number of $\vf(\ov x,\ov y)$ is at least one.
\end{proof}
\begin{defn}
We call a formula with dual-alternation number zero \textbf{non-alternating}. We will say that the formula is NA or a NA-formula.
\end{defn}
The following definition comes from  Adler \cite{Adl:VCmin}.
\begin{defn}\label{def: directed family}
  We call a set of partitioned formulas $\Psi(\ov x,\ov y)$ an instantiable directed family (or a directed family for short) if for any two sets $A,B$ defined by instances of formulas in $\Psi$  either $A\se B$, $B\se A$, or $A\cap B=\buit$. By an instance of $\vf(\ov x,\ov y)$ we mean the formula $\vf(\ov x,\ov b)$ for some $\ov b$ of the appropriate length.
\end{defn}
In the previous definition $\ov x$ is a finite tuple of variables, while we allow $\ov y$ to be an infinite tuple. Hence, although we use the same tuple $\ov y$, not all formulas in $\Psi$ must have the same finite tuple of parameter variables.

\begin{prop}\label{prop: single-row}
Let $\Psi$ be a directed family of partitioned formulas. 
Then there is no wict-pattern with two different rows  with formulas of $\Psi$ or negation of formulas of $\Psi$ (two formulas of the sequence with different index cannot both belong to $\Psi\cup\lnot\Psi$, where  $\lnot\Psi$ denotes  $\set{\lnot \psi \mid \psi\in\Psi}$).
\end{prop}
Obviously the same holds for ict or special patterns.
\begin{proof}

Assume  there is a wict-pattern of depth 2 where the two formulas of the sequence or their negations belonging to $\Psi$. We may assume the pattern is of kind $2\times \om$ with \mui\ rows. Let us denote by $(A_i \mid i\in\om)$ the sets defined by the first row and $(B_i \mid i\in\om)$  the sets defined by the second. We have to distinguish three cases depending on whether the formulas of the pattern or their negations belong to $\Psi$.

Case 1. Both formulas belong to $\Psi$. By the consistency of the path $(0,0)$, $A_0$ and $B_0$ have nonempty intersection, thus \wlogy\ $A_0\se B_0$. By \muy\ $A_0\se B_1$, which contradicts the consistency of the path $(0,0)$.

Case 2. Only one formula belongs to $\Psi$. We may assume the $A_i$ are defined by instances of formulas in $\Psi$ while $B_i=D_i^C$ is the complement  of a set $D_i$ defined by an instance of a formula in $\Psi$. Again by the consistency of the path $(0,0)$, $A_0\cap B_1^C=A_0\cap D_1$ is nonempty, hence either $A_0\se D_1$ or $D_1\se A_0$. In the first case, by \muy, $A_0\se D_0$ and thus $A_0\cap B_0=\buit$, which contradicts the consistency of the the path $(0,0)$. The second case  implies $B_1^C\se A_1$. This contradicts again the consistency of the path $(0,0)$ because $A_1^C\cap B_1^C$ should be nonempty.

Case 3. No formula belong to $\Psi$. In this case both $A_i$ and $B_i$ are the complements of $C_i$ and $D_i$ respectively, sets defined by instances of formulas in $\Psi$. Again the consistency of the path $(0,0)$ implies $A_1^C\cap B_1^C=C_1\cap D_1$ is nonempty. Hence we may assume $C_1\se D_1$ and thus $B_1\se A_1$. By \muy, $B_0\se A_1$ which contradicts again the consistency of the path $(0,0)$.
\end{proof}

\begin{defn}
Let $\la$ be a cardinal (finite or infinite). We say that a complete theory $T$ is $\la$-VCA if there is collection of $\la$ instantiable directed families $\gen{\Psi_i(x,\ov y)\mid i\in\la}$  such that each definable 1-set in the single variable $x$ is a boolean combination of instances of formulas in $\bigcup_{i\in\la}\Psi_i(x,\ov y)$ and instances of NA-formulas of kind $\psi(x,\ov y)$. We call a complete theory $T$  VCA-minimal if it is $1$-VCA.
 \end{defn}

Obviously any VC-minimal theory is VCA-minimal. However the converse fails as Proposition \ref{prop: caract. dp-minimal} shows. 

\begin{fet}\label{fet: finite boolean options}
Let $\Si(\ov x,\ov y)$ be a set of formulas without parameters, where the tuple $\ov y$ can be infinite. Assume each definable set with parameters with free variables among $\ov x$ is definable by an instance of some formula in $\Si$. Then, for each formula $\vf(\ov x,\ov y)$ without parameters there is a finite subset $\Theta$ of $\Si$ such that each instance of $\vf(\ov x,\ov y)$ is an instance of some formula in  $\Theta$.
\end{fet}

\begin{proof}
$$
T\cup\set{\lnot\exists\ov z\forall x(\vf(x,\ov y)\leftrightarrow\psi(x,\ov z))  ,\psi(\ov x,\ov y)\in \Sigma  }
$$
is inconsistent.
By compactness $T\cup\set{\lnot\exists\ov z\forall x(\vf(x,\ov y)\leftrightarrow\psi(x,\ov z))  ,\psi(\ov x,\ov y)\in \Theta  }
$ is inconsistent  for some finite $\Theta\se \Sigma$.
\end{proof}

\begin{cor}\label{cor: VCA-minimal => dp-minimal}
Any VCA-minimal theory is \textrm{dp}-minimal.
\end{cor}
\begin{proof}
Let $\Psi( x,\ov y)$ be a directed family witnessing the VCA-minimality of the theory. We may assume that each formula in any wict-pattern for $x=x$ is a boolean combination of formulas of $\Psi$ and NA-formulas: in each row of the pattern some boolean combination provided by Fact \ref{fet: finite boolean options} will occur infinitely many times.
By Propositions \ref{prop: no disjunctions} and \ref{prop: no conjunctions} we can assume each row of a wict-pattern is constituted by a formula of $\Psi$, a negation of a formula of $\Psi$ or a non-alternating formula (the negation of an NA-formula is NA). By Proposition \ref{prop: single-row} $\Psi$ can contribute with at most one row, and no non-alternating formula can occur in a wict-pattern.
\end{proof}
We will see in Proposition \ref{prop: caract. dp-minimal} that any \textrm{dp}-minimal \oag\ is VCA-minimal.

In fact, this is more general:

\begin{prop}\label{prop: la-VCA implica dprang le la}
If $T$ is $\la$-VCA then the $\dpr(T)\le\la$.
\end{prop}
\begin{proof}
As in Corollary \ref{cor: VCA-minimal => dp-minimal} using propositions \ref{prop: no conjunctions}, \ref{prop: no disjunctions}, \ref{prop: single-row} and the fact that NA-formulas cannot occur in a wict-pattern.
\end{proof}

%% file: Dp-minimalOAG.tex
\section{dp-minimal ordered groups}\label{sec: dp-minimal ordered groups}

Here we work again with the monster model of some complete theory.



The following is not the standard definition of the dual VC-density, but equivalent to it. See \cite{ADHMS:2}.

\begin{defn}\label{def: dual VC-density}
  The dual VC-density of a partitioned formula $\vf(\ov x,\ov y)$, denoted by $vc^*(\vf(\ov x,\ov y))$, is the infimum of all real numbers $r>0$ such that for any finite set $A$ of $\abs{\ov y}$-tuples $\abs{S^\vf (A)}=O(\abs{A}^r)$. Here $S^\vf (A)$ denotes the set of all maximally consistent sets of formulas of kind $\vf(\ov x,\ov a)$ or $\lnot\vf(\ov x,\ov a)$ with $\ov a\in A$.
\end{defn}

For more details on the dual VC-density, see \cite{ADHMS:2}.

In the proof of next proposition we use the following facts:

\begin{fets}\label{fets: sobre el cardinal de S^varphi(A) i els atoms d'una algebra de boole finita}
\begin{enumerate}
  \item $\abs{S^\vf (A)}$ coincides with the number of atoms of the Boolean algebra of sets generated by the definable sets $\vf(\C,\ov a)$ with $\ov a\in A$. Here $\C$ denotes the monster model (or any model) of the theory and $\vf(\C,\ov a)$ denotes the set defined by $\vf(\ov  x,\ov a)$ in $\C$.
  \item Let $B$ be a Boolean algebra and $B_1,B_2$ be given finite subalgebras (closed under $\land$ and $\lnot$) of $B$. Let us denote by $at(B_i)$ the number of atoms of $B_i$. Then the number of atoms of the subalgebra generated by $B_1\cup B_2$ is at most $at(B_1)at(B_2)$.
  \item Let $\tira An$ be a directed family of sets, i.e. if $A_i\cap A_j\ne\buit$ then either $A_i\se A_j$ or $A_j\se A_i$. Then the Boolean algebra of sets generated by $\tira An$ has at most $n+1$ atoms.
\end{enumerate}
\end{fets}

\begin{prop}\label{prop: VCA minimal implica 1-vc-density=1}
In a VCA-minimal theory any formula $\vf(x,\ov y)$ has dual VC-density at most 1.
\end{prop}
\begin{proof} By fact \ref{fet: finite boolean options}, for each formula $\vf(x,\ov z)$ there is a finite set $\Theta$ of formulas of kind $\mu(x,\ov u)$ which are Boolean combinations of formulas in $\Psi$ and NA-formulas, such that each instance of $\vf(x,\ov z)$ is equivalent to an instance of some formula in $\Theta$. Let $\Psi_0$ denote respectively the set of formulas from $\Psi$ occurring in the boolean combinations of formulas in $\Theta$ and let $\Upsilon$ be the set of NA-formulas occurring in the boolean combinations of formulas in $\Theta$. Let $N$ be a common upper bound of the number of different sets each formula in $\Upsilon$ can define. We may assume all formulas in $\Psi,\Theta$ and $\Upsilon$ have the same parameter variables, say $\ov u$.

Given a set of $\ov z$-parameters $A$, we can choose a set of  $\ov u$-parameters $B$ of size at most $\abs{A}$ such that each instance of $\vf(x,\ov z)$ with parameters from $A$ is an instance of some formula in $\Theta$ with parameters from $B$. Hence, any definable set $\vf(\C,\ov a)$ with $\ov a\in A$ is a boolean combination of sets of kind $\psi(\C,\ov b)$ where $\psi(x,\ov u)\in \Psi\cup\Upsilon$ and $\ov b\in B$.

Now it is not difficult to see that $\abs{S_{\vf}(A)}\le
(\abs{\Psi_0}\abs{A}+1)2^{N^{\abs{\Upsilon}}} $. This holds because the boolean algebra generated by the sets defined by $\Psi_0$-formulas with parameters from $B$ has at most $\abs{\Psi_0}\abs{A}+1$ atoms. And the Boolean algebra generated by $\Upsilon$-formulas with parameters in $B$  has at most $2^{N^{\abs{\Upsilon}}}  $ atoms becase there are at most $N^{\abs{\Upsilon}}$ such nonequivalent formulas.
\end{proof}

Here is the complete characterization of all \textrm{dp}-minimal\ ordered groups.

We say that an \oag\ has \textbf{finite dimension} iff $\dimp(G)$ is finite for all prime $p$.

\begin{prop}\label{prop: caract. dp-minimal}
Let $G$ be an ordered group. \Tfae
\begin{enumerate}
  \item $G$ is VCA-minimal.
  \item Any formula $\vf(x,\ov y)$  has dual VC-density at most 1 (in the theory of $G$).
  \item $G$ is\ \textrm{dp}-minimal.
  \item $G$ is Abelian and has finite dimension. 
\end{enumerate}
\end{prop}
\begin{proof}
\emph{1} implies \emph{2} By proposition \ref{prop: VCA minimal implica 1-vc-density=1}.

\emph{2} implies \emph{3}. By \cite{DGL} Proposition 3.2.

\emph{3} implies \emph{4}.
$G$ is abelian by Proposition 3.3 of \cite{Sim:Dp-min}.
Assume now $G$ is Abelian but $\dim_p(G)$ is infinite for some prime $p$. As $[G:pG]$ is infinite, let $(b_j)_{j\in\om}$ be an infinite set of positive elements non congruent modulo $pG$. Let $a$ be an element greater than any $b_j$ (in the monster model!).   Then the following is a wict-pattern of depth 2: $(x>ipa \mid i\in\om)$, $(x\equiv b_i \mod p \mid i\in\om)$.

\emph{4} implies \emph{1}.
Now assume that for each prime number $p$, $\dim_p(G)$ is finite. Then, by Remark 5 in \ref{rem: rang n-regular finit}, each $p$-regular rank should be finite.  By Theorem \ref{thm: QE}, any formula can be written as a boolean combination of formulas of the following kind:
\begin{eqnarray}
\label{eq: atoms proof strong 01}
  nx &\le & t(\ov y) \\
\label{eq: atoms proof strong 02}
  nx &\equiv & t(\ov y) \mod \De    \\
\label{eq: atoms proof strong 03}
  nx &\equiv & t(\ov y) \mod \De +p^mG   \\
\nonumber
\end{eqnarray}
where $n\in \Z$, $\De\in\rjp(G),\ p\in\PP$,  $m\ge 1$ and $t(\ov y)$ is a term.

Since the formulas of kind \eqref{eq: atoms proof strong 01} and \eqref{eq: atoms proof strong 02} define convex subsets, they are  boolean combination of definable initial segments. So we may express any formula as a boolean combination of definable initial segments and formulas of kind \eqref{eq: atoms proof strong 03}. As the formulas defining initials segments constitute a directed family it only remains to prove that formulas of kind \eqref{eq: atoms proof strong 03} are NA.

\begin{claim}\label{cl: cosets Delta + p^kG}
Every nonempty instance of the formula $nx\equiv t(\ov y) \mod \De+p^mG$ is a coset of $\De+p^rG$, where $n=n'p^s$, $\gcd(n',p)=1$ and $r=\min\set{m-s,0}$.
\end{claim}
\begin{proof}
easy.
\end{proof}
Finally
 \begin{claim}\label{cl: p-dim finita implica la formula .... es NA}
 If $\dim_p(G)$ is finite, the formula
 $$nx\equiv t(\ov y)\mod \De+p^mG$$ is  NA.
 \end{claim}
 \begin{proof}
By claim \ref{cl: cosets Delta + p^kG}, it suffices to show $[G:\De+p^rG]$ is finite.

Since $G/\De\Big/p(G/\De)\simeq G/(\De+pG)$ is a free $\F_p$-module of rank $\dim_p(G/\De)$ then  $G/\De\Big/p^r(G/\De)\simeq G/(\De+p^rG)$ is a free $\Z\big/p^r\Z$-module of rank $\dim_p(G/\De)$ (see \cite{Kap}). This implies $[G:\De+p^rG]=p^{r\dim_p (G/\De)}\le p^{r\dim_p (G)}$, which is finite.
\end{proof}
\end{proof}
Observe that Corollary \ref{cor: VCA-minimal => dp-minimal} gives another proof that \emph{1} implies \emph{3}.

We cannot replace VCA-minimal by VC-minimal. In \cite{Flenn-Guin} it is shown that any convexly orderable \oag\ is divisible. Moreover any VC-minimal theory is convexly orderable (see \cite{Gui-Las} for a proof of this). This shows that there are many VCA-minimal \oag s non convexly orderable.

%% file: ChainConditions.tex
\section{Chain conditions}\label{sec: Chain conditions}

We start with Lemma \ref{lem: TXR en grups}. It is a result about pure groups and it is a generalization of the Chinese Remainder Theorem to the non-Abelian case. We assume there is some group $G$, $a,b$ are elements of $G$ and $H$ is a subgroup of $G$. We use the congruence notation to work with left cosets: $a\equiv b \bmod H$ means $aH=bH$.

If $H_1$ and $H_2$ are subgroups of $G$, we let $H_1H_2$ denotes the following  set: $\set{h_1h_2\mid h_1\in H_1, h_2\in H_2}$. Obviously $H_1H_2$ is a subgroup iff $H_1H_2=H_2H_1$.  Observe that the condition \eqref{eq: ccc} depends on the order of the sequence of groups, while the conclusion not.
\begin{lem}\label{lem: TXR en grups}
Let $G$ be a group, $\tira Hn$ a sequence of subgroups of $G$ satisfying the following:

\begin{equation}\label{eq: ccc}
    \bigcap_{i<r}(H_iH_r) =\left( \bigcap_{i<r}H_i\right)H_r \text{ \ \  for }r=2\ldots n
\end{equation}
Then, given $\tira an\in G$, the system
\begin{equation}\label{eq: sistema xines, not. congruencies}
    \begin{cases}
    x \equiv a_1 & \bmod\ H_1 \\
  \hfill  \vdots &  \\
    x \equiv a_n & \bmod \ H_n \\
    \end{cases}
\end{equation}
has a solution iff for any $i<j\le n$ it happens that $a_i^{-1} a_j\in H_iH_j$.

Moreover, if $b$ is a solution of \eqref{eq: sistema xines, not. congruencies}, the system \eqref{eq: sistema xines, not. congruencies} is equivalent to
$$
x\equiv b \mod \bigcap_{i\le n}H_i.
$$
\end{lem}
\begin{proof}
If $c$ is a solution of \eqref{eq: sistema xines, not. congruencies} then $a_i^{-1}c\in H_i$ and $a_j^{-1}c\in H_j$, so $a_i^{-1}a_j\in H_iH_j$.
If $b$ is a particular solution of \eqref{eq: sistema xines, not. congruencies}, $c$ is a solution of \eqref{eq: sistema xines, not. congruencies} iff $c\equiv b\mod H_i$ for all $i$ iff $c\equiv b\mod \bigcap_{i\le n}H_i$.
The proof of the existence of a solution in case the compatibility condition is satisfied, is by induction on $n$.

Case $n=2$.  If $a_1^{-1} a_2\in H_1H_2$ then $a_1^{-1}a_2=h_1h_2$ for some $h_i\in H_i$. Hence $b=a_1h_1=a_2h_2^{-1}$ is a solution.

Now we prove the inductive step. Assume the compatibility conditions $a_i^{-1} a_j\in H_iH_j$ holds for any $i<j\le n+1$. By inductive hypothesis, the system of the first $n$ congruencies has a solution, say $b$. It suffices to prove that the system
\begin{equation}\label{eq: sistema inductiu 2}
    \begin{cases}
    x \equiv b & \mod \bigcap_{i\le n}H_i \\
    x \equiv a_{n+1} & \mod H_{n+1} \\
    \end{cases}
\end{equation}
Has a solution. By the case with $n=2$ it suffices to show that $b^{-1}a_{n+1}\in\left(\bigcap_{i\le n}H_i\right)H_{n+1}$. Since $b$ is a solution of the first $n$ congruencies $b=a_ih_i$ for some $h_i\in H_i$ and for any $i\le n$. Hence $b^{-1}a_{n+1}=(h_i^{-1}a_i^{-1}a_{n+1})\in H_iH_{n+1}$. Therefore $b^{-1}a_{n+1}\in \bigcap_{i\le n}(H_iH_{n+1})=\left(\bigcap_{i\le n}H_i\right)H_{n+1}$ by \eqref{eq: ccc}.
\end{proof}

\begin{rem}In the previous lemma, if moreover for any $1\le i<j\le n$ one gets $H_iH_j=H_jH_i$, then all `sets' in the statement are subgroups: $H_iH_j$, and $\left(\bigcap_{i<r}H_i\right)H_r$ \ for $r=2\ldots n$. Moreover$\left(\bigcap_{i<r}H_i\right)H_r=H_r\left(\bigcap_{i<r}H_i\right)$ for $r=2\ldots n$.
\end{rem}

The main results of this section are Propositions \ref{prop: cota inferior burden en grups, cas finit}, \ref{prop: cota inferior burden en grups, cas numerable} and \ref{prop: cota inferior burden en grups, cas infinit mateixa formula}.They are stated in terms of burden, but one can replace burden by \dpra\ and remain true. This is because the burden is always smaller or equal to  \dpra. More precisely, if a certain type has an inp-pattern of a certain depth the same type also has an ict-pattern of the same depth. But in fact, in the proofs, we exhibit patterns which are ict and inp at the same time. For the sake of completeness we include a definition of inp-pattern and burden. Here again we work in the monster model of some complete theory. It is aso known that burden and \dpra\ coincide for NIP types, see \cite{Adl:Strong}.
\begin{defn}
Let $p(\ov x)$ be a partial type. An \textbf{inp-pattern} for $p(\ov x)$ consist in the following data: a sequence of formulas $(\vf(\ov x,\ov y_i) \mid i\in k)$, a sequence of natural numbers $(k_i \mid i\in k)$ and an array $A:=(\ov a_j^i \mid i\in\ka,\ j\in O)$, where $\ka$ is a cardinal number and $O$ is an infinite linearly ordered set \st:\begin{itemize}
      \item for every $f\in O^\ka$, the following set of formulas is consistent: $$p(\ov x)\cup
\set{\vf_i(\ov x,\ov a_{f(i)}^i) \mid i\in \ka}$$
      \item for each $i\in\ka$ the set $\set{\vf(\ov x,a_i^j)\mid j\in O}$ is $k_i$-inconsistent.
    \end{itemize}
The cardinal number $\ka$ is called the depth of the pattern and we allow $\ka$ to be finite. We also say that it is an inp-pattern of type $\ka\times O$. 

The \textbf{burden} of $p(\ov x)$ denoted by $\bdn(p(\ov x))$ is the supremum of all depths of all inp-patterns for $p(\ov x)$.

If $T$ denotes a complete theory with a main sort \footnote{there is a sort, called main sort, such that all other sorts are obtained as sorts of imaginaries of the theory of the main sort alone} the burden of $T$, denoted by  $bdn(T)$ is  $\bdn(x=x)$ where $x$ is a single variable for the main sort.

 A complete theory $T$ is called \textbf{strong} (see \cite{Adl:Strong}) iff there is no inp-pattern of infinite depth for the formulas $\ov x=\ov x$, where $\ov x$ is of finite length. By the submultiplicativity of burden (see \cite{Chernikov-NTP2}) it is enough to consider a single variable.

Shelah in \cite{Sh:863} calls $T$ strongly dependent iff there in no ict-pattern of infinite depth for the formulas $\ov x=\ov x$ with $\ov x$ of finite length. It is enough again to consider a single variable. A theory is strongly dependent iff it is NIP and strong, see \cite{Adl:Strong}. As all \oag s are NIP (see \cite{Gu-Schm}) strong and strongly dependent are equivalent properties for \oag s.

When $\dpr(T)$ is finite we say that $T$ has \textbf{finite \dpra} (in \cite{Kaplan-Shelah:Chain2013} it is said that the theory has bounded \dpra). Obviously any theory of finite \dpra\ is strong. 
In section \ref{sec: Strong oag s} we will see that any strong \oag\ has finite \dpra.

\end{defn}

\begin{prop}\label{prop: cota inferior burden en grups, cas finit}
Let $G$ be an expansion of a group. Assume there is a sequence of definable subgroups (maybe with parameters) $\tira Hn$ satisfying the following three conditions:
\begin{description}
  \item[commutativity:]   $H_iH_j=H_jH_i$ for $1\le i<j\le n$.
  \item[distributivity:]    $ \bigcap_{i<r}(H_iH_r)=\left( \bigcap_{i<r}H_i\right)H_r \text{ for }r=2\ldots n$
  \item[infinity:] $[K_i:H_i]$ is infinite for $r=1\ldots n$.
\end{description}
where $K_i$ denotes $\ds\bigcap_{\substack{j=1\\ j\ne i}}^nH_jH_i$.

Then $\bdn(G)\ge n$.

\end{prop}
\begin{proof}
 By \textbf{infinity} we can choose $ a_j^i$, for $1\le i\le n,\ j\in\om$ \st $a_j^i\in K_i$ but $a_j^i\not\equiv a_{j'}^i \mod H_i$ for $j<j'$. Lets check that the formulas $x\equiv  a_j^i \mod H_i$ constitute an inp-pattern with all rows 2-inconsistent. Given $f\in\om^n$, we apply Lemma \ref{lem: TXR en grups} to check that The following system
\begin{equation}\label{eq: sistema xines 2}
    \begin{cases}
    x \equiv a_{f(1)}^1 & \mod H_1 \\
  \hfill  \vdots &  \\
    x \equiv a_{f(n)}^n & \mod H_n \\
    \end{cases}
\end{equation}
has a solution.  Observe that for $i\ne j$,  $a_{f(i)}^i\in K_i\se H_jH_i$ and $a_{f(j)}^j\in K_j\se H_iH_j$ implies ${a_{f(i)}^i}^{-1}a_{f(j)}^j\in H_iH_j$ because of the commutativity condition.
\end{proof}

\begin{prop}\label{prop: cota inferior burden en grups, cas numerable}
Let $G$ be an expansion of a group. Assume there is a sequence of definable subgroups (maybe with parameters) $(H_i\mid i\in \om)$ satisfying the following three conditions:
\begin{description}
  \item[commutativity:]   $H_iH_j=H_jH_i$ for $i,j\in\om$.
  \item[distributivity:]    $ \bigcap_{i<n}(H_iH_n)=\left( \bigcap_{i<n}H_i\right)H_n \text{ for }n\in \om$
  \item[infinity:] $[K^n_i:H_i]$ is infinite for $i\le n\in\om$.
\end{description}
where $K^n_i$ denotes $\ds\bigcap_{\substack{j=1\\ j\ne i}}^nH_jH_i$.

Then $G$ is not strong.
\end{prop}
\begin{proof}
 Consider the following sequence of formulas for a pattern of depth $\aleph_0$:
$$
(x\equiv y \mod H_i, i\in\om).
$$
A compactness argument shows that, in order to find an array of $y$-parameters making this sequences an inp-pattern (with all rows 2-inconsistent), it is enough to find parameters making the first $n+1$ formulas an inp-pattern  for each $n$ (with all rows 2-inconsistent). Now, the consistency corresponding to the first $n+1$  rows is obtained by the same arguments as in Proposition \ref{prop: cota inferior burden en grups, cas finit}.

\end{proof}

It is interesting to compare this result with Proposition 3.12 in \cite{Kaplan-Shelah:Chain2013}.

A complete theory is called NTP2 iff there is no inp-pattern of infinite depth with the same formula in all rows (the sequence of formulas of the pattern is constant) for the type $\ov x=\ov x$, where $\ov x$ is a finite tuple of variables.

\begin{prop}\label{prop: cota inferior burden en grups, cas infinit mateixa formula}
Let $G$ be an expansion of a group. Assume there is a uniformly definable sequence of subgroups $(H_i\mid i\in \om)$ satisfying the following three conditions:
\begin{description}
  \item[commutativity:]   $H_iH_j=H_jH_i$ for $i,j\in\om$.
  \item[distributivity:]    $ \bigcap_{i<n}(H_iH_n)=\left( \bigcap_{i<n}H_i\right)H_n \text{ for }n\in \om$ite tuple of variables.
  \item[infinity:] $[K^n_i:H_i]$ is infinite for $i\le n\in\om$.
\end{description}
where $K^n_i$ denotes $\ds\bigcap_{\substack{j=1\\ j\ne i}}^nH_jH_i$.

Then  $G$ is TP2.
\end{prop}
\begin{proof}
As in Proposition \ref{prop: cota inferior burden en grups, cas numerable}. Observe that given the infinite family of $(H_i\mid i\in\om)$ we obtain a fixed  formula for a pattern of depth $\aleph_0$.
\end{proof}

%% file: StrongOAG.tex
\section{Strong \oag s 
}\label{sec: Strong oag s}

In this section we characterize strong \oag s. As a corollary we will see that a strong \oag\ has finite \dpra. 

When $\set{p\in\PP \mid \dim_p(G)\ge\aleph_0}$ is finite we will say that $G$ has \textbf{almost finite dimension}.

We begin by something easy:
\begin{prop}\label{prop: infinite p-dim=inf implica no strong}
 Any strong \oag\  has almost finite dimension. 

\end{prop}
\begin{proof}
If $\set{p\in\PP \mid \dim_p(G)\ge\aleph_0}$ is infinite, apply Proposition \ref{prop: cota inferior burden en grups, cas numerable} with $H_i=p_iG$, where $(p_i\mid i\in\om)$ denotes the sequence of all $p$ with $\dimp(G)$ infinite. Observe that $K_i^n=G$. Then $G$ is not strong.

\end{proof}

The proof of this lemma shows that in fact $$\dpr(G)\ge\abs{\set{p\in\PP \mid \dim_p(G)\ge\aleph_0}},$$ but we will obtain better bounds in \ref{thm: formula Dp-rang, cas RR acotat}.

\begin{thm}\label{thm: p-regular rank infinit implica no strong}
Any  strong \oag\ has bounded regular rank.
\end{thm}
\begin{proof}
We may assume $G$ is enough saturated. Assume $G$ has infinite $p$-regular rank for some prime number $p$.

By saturation, there is an infinite ascending chain $\De_0\lhd \De_1\lhd\cdots \De_i \lhd \De_{i+1}\lhd\cdots$, where $\De_i=A_p(a_i)$ and $\gen{a_i\mid i\in \om}$ is a sequence of elements of $G$.
By deleting some elements we may assume $\dimp(\De_{i+1}/\De_i)>N$ for any $i\in\om$. This is because $\dimp(\De_{i+2}/\De_i)>0$. By saturation again we may assume $\dimp(\De_{i+1}/\De_i)\ge\aleph_0$.

Now consider the following family of definable groups:
\begin{equation}\label{eq: Hs}
    H_i:=\De_i+p^iG\text{ \ for \ }i\ge1.
\end{equation}
 Our aim is to show that this family satisfies the hypothesis of Proposition \ref{prop: cota inferior burden en grups, cas numerable}.
 This follows from the following claim. Observe that \emph{3} together with \emph{4} imply distributivity and \emph{6} and \emph{7} infinity.
 \begin{claim}
 \begin{enumerate}
   \item If $i<j$ then $H_i+H_j=\De_j+p^iG$.
   \item If $i<j$ then $H_i\cap H_j=\De_i+p^i\De_j+p^jG$.
   \item $H_1\cap\cdots\cap H_r=\De_1+ p\De_2+\cdots +p^{r-1}\De_r + p^rG$.
   \item  $\bigcap_{i<r}(H_i+H_r)=\De_r+p^{r-1}G$.
   \item If $r<n$, then $\ds K_r^n:=\bigcap_{\substack{i=1\\ i\ne r}}^n(H_i+H_r)=\De_r+p^{r-1}\De_{r+1} +p^rG$.
   \item If $r<n$, then $[K_r^n:H_r]$ is infinite.
   \item If $K_n^n:=\bigcap_{i<n}(H_i+H_n)$ then  $[K_n^n:H_n]$ is infinite.
 \end{enumerate}
\end{claim}
\begin{proof}
\emph{1} is because $\De_i\subset \De_j$ and $p^jG\subset p^iG$.

\emph{2}. If $\de_i+p^ig=\de_j+p^jh$ for some $\de_i\in\De_i$, $\de_j\in\De_j$, $g,h\in G$, then $\de_j-\de_i\in p^i\De_j$ whence $\de_j+p^jh\in \De_i+p^i\De_j+p^jG$.

\emph{3} is done by induction on $r$. Assume $H_1\cap\cdots\cap H_r=\De_1+ p\De_2+\cdots +p^{r-1}\De_r + p^rG$. In order to prove $H_1\cap\cdots\cap H_{r+1}=\De_1+ p\De_2+\cdots +p^{r}\De_{r+1} + p^{r+1}G$ it suffices to see $\left(\De_1+ p\De_2+\cdots +p^{r-1}\De_r + p^rG\right)\cap H_{r+1}\subseteq\De_1+ p\De_2+\cdots +p^{r}\De_{r+1} + p^{r+1}G$, the other inclusion being obvious. If $\de_1+p\de_2+\cdots+p^{r-1}\de_r+p^rg=\de_{r+1}+p^{r+1}h$ then $\de_{r+1}-(\de_1+p\de_2+\cdots+p^{r-1}\de_r)\in p^r\De_{r+1}$ an thus $\de_{r+1}\in \De_1+ p\De_2+\cdots +p^{r}\De_{r+1}$. Therefore  $\de_{r+1}+p^{r+1}h\in\De_1+ p\De_2+\cdots +p^{r}\De_{r+1} + p^{r+1}G$.

\emph{4} follows from $H_j+H_r=\De_r+p^jG\subseteq \De_r+p^iG=H_i+H_r$ for $i<j< r$.

\emph{5}. Since for $r<i<j$, $H_i+H_r=\De_i+p^rG\subseteq \De_j+p^rG=H_j+H_r$, obviously $\bigcap_{r<i\le n}^n(H_i+H_r)=\De_{r+1}+p^rG$.
It remains to show $(\De_r+p^{r-1}G)\cap(\De_{r+1}+p^rG)\se \De_r+p^{r-1}\De_{r+1} +p^rG$, as the converse inclusion is obvious. If $\de_r+p^{r-1}g=\de_{r+1}+p^rh$ then $\de_{r+1}-\de_r\in p^{r-1}\De_{r+1}$ and thus $\de_{r+1}+p^rh\in \De_r+p^{r-1}\De_{r+1} +p^rG$.

\emph{6}. We show that $K_r^n/H_r\simeq \De_{r+1}/(\De_r+p\De_{r+1})$. Then $[K_r^n:H_r]$ is infinite because $\De_{r+1}/(\De_r+p\De_{r+1})\simeq \left(\De_{r+1}/\De_r\right)/p\left(\De_{r+1}/\De_r\right)$ and  $\dim_p(\De_{r+1}/\De_r)$ is infinite. This is done by checking that the morphism $\De_{r+1}\rightarrow K_r^n/H_r$ given by $x\mapsto p^{r-1}x+H_r$ is surjective and its kernel is $\De_r+p\De_{r+1}$.

\emph{7}. Observe that $K_n^n=\De_n+p^{n-1}G$. We show that $K^n_n/H_n\simeq G/(\De_n+pG)$. This is done by checking that  the morphism $G\rightarrow (\De_n+p^{n-1}G)/(\De_n+p^nG)$ given by $x\mapsto p^{n-1}x+H_n$ is surjective and its kernel is $\De_n+pG$. Then $[K^n_n:H_n]$ is infinite because $G/(\De_n+pG)\simeq \left(G/\De_n\right)/p\left(G/\De_n\right)$ and  $\dim_p(G/\De_n)$ is infinite.
\end{proof}
\end{proof}

\begin{thm}\label{thm: cota superior Dp-rang, cas strong} Any \oag\ with bounded regular rank and almost finite dimension has finite \dpra.
\end{thm}
\begin{proof}
As in the proof of  Proposition \ref{prop: caract. dp-minimal}, By \ref{thm: QE}, \ref{prop: wict and special patterns compute dp-rank}, \ref{prop: no disjunctions} and \ref{prop: no conjunctions}, we may assume  the formulas of the pattern either  define initial segments or are  of kind \eqref{eq: atoms proof strong 03}. As the formulas defining initials segments constitute a directed family, by Proposition \ref{prop: single-row} they can only contribute with one row to the pattern.

Given a prime number  $p$ and $\De\in\rjp(G)$, there is at most one row with  formulas of kind \eqref{eq: atoms proof strong 03}  or its negation with this $\De$ and $p$. The reason is that these formulas (with the same $\De$ and $p$ and maybe different $m$), by Claim \ref{cl: cosets Delta + p^kG} constitute again a directed family.

Moreover, if $\dim_p(G)$ is finite, the formulas of kind \eqref{eq: atoms proof strong 03} for $\De\in\rjp(G)$ cannot occur (nor their negations) in the pattern because, by Claim \ref{cl: p-dim finita implica la formula .... es  NA}, they are NA-formulas. Hence, as $\set{p\in\PP \mid \dim_p(G)\ge\aleph_0}$ is finite and $\rjp(G)$ are finite for each prime $p$ we have only a finite number of possible rows.
\end{proof}

In fact, the proof of Theorem \ref{thm: cota superior Dp-rang, cas strong} provides us the following upper bound for
$\dpr(G)$:

$$
\dpr(G)\le1+\sum_{p\in\PP \atop \dim_p(G)\ge\aleph_0}\abs{\rjp(G)}
$$
Observe that this formula implies that any \oag\ with finite dimension is dp-minimal.
In the next section we will improve this formula in order to get the exact value of $\dpr(G)$.

The main Theorem of this section is the characterization of strong \oag s:
\begin{thm}\label{thm: caracteritzacio strong}
An \oag\ is strong iff it has bounded regular rank and almost finite dimension.
\end{thm}
\begin{proof}
  By Proposition \ref{prop: infinite p-dim=inf implica no strong}, Theorem \ref{thm: p-regular rank infinit implica no strong}  and Theorem \ref{thm: cota superior Dp-rang, cas strong}.
\end{proof}

\begin{cor}\label{cor: strong OAG => finite dp-rank}
Any strong \oag\ has finite \dpra.
\end{cor}

\begin{proof}
  By Proposition \ref{prop: infinite p-dim=inf implica no strong}, Theorem \ref{thm: p-regular rank infinit implica no strong}  and Theorem \ref{thm: cota superior Dp-rang, cas strong}.
\end{proof}

%% file: ComputingDpRankInOAG.tex
\section{Computing the  dp-rank in \oag s}\label{sec: computing dpra in oag}

In this section we push forward the argument of the previous section and compute the precise value of the \dpra\ of $G$. This value depends on the number of `infinite' regular jumps. Given a prime number $p$, and a $p$-regular jump $\De\in\rjp(G)$, we say that $\De$ is `infinite', if $\dim_p(\De'/\De)$ is infinite where $\De'$ is the successor of $\De$ in $\rjp(G)$ (or $\De'=G$ in case $\Delta$ has no successor). In fact, we are going to see that for a non-trivial \oag\ $G$, the \dpra\ of $G$ equals 1+plus the number of infinite $p$-regular jumps for all different $p$.

 But one must be carefully with this not very precise statement. It can happen that the same convex subgroup $\De$ can contribute as an infinite regular jump for different primes. Then we must count the number of primes in which $\De$ occurs, i.e. count the the `multiplicity'.
To be more precise let use introduce a little bit more notation. If $\rjp(G)=\set{\tiraa{\De}{0}{{u-1}}}$,  then
$$
\rjpi(G):=\set{\De_i\in \rjp(G) \mid \dimp(\De_{i+1}/\De_i)=\infty},
$$
where $\De_{u}$ denotes $G$.
We call $\rjpi(G)$ the set of infinite $p$-regular jumps.

\begin{rem}\label{rem: p-dim finita <=> p-rr finit + RJpinf buit} $G$ has finite $p$-dimension iff $G$ has finite $p$-regular rank and $\rjpi=\buit$.
\end{rem}
\begin{proof}
By point \emph{5} in remark \ref{rem: rang n-regular finit} and the additivity of the $p$-dimension:\newline
$\dim_p(G)=\dim_p(G/\De_{u-1})+\dim_p(\De_{u-1}/\De_{u-2})+\cdots +\dim_p(\De_{1}/\De_{0}).$
\end{proof}

\begin{thm}\label{thm: formula Dp-rang, cas RR acotat}
Let $G$ be a nontrivial \oag\ with bounded regular rank. Then
\begin{equation}\label{eq: formula Dp-rang OAG}
    \dpr(G)=1+\sum_{p\in\PP}\abs{\rjpi(G)}
\end{equation}
\end{thm}
\begin{proof}
We begin by showing the inequality $\dpr(G)\ge 1+\sum_{p\in\PP}\abs{\rjpi(G)}.$
We may assume that $\sum_{p\in\PP}\abs{\rjpi(G)}$ is finite by Proposition \ref{prop: infinite p-dim=inf implica no strong} and Remark \ref{rem: p-dim finita <=> p-rr finit + RJpinf buit}.

Let us denote $\tira{\De}N$ all convex subgroups occurring in  $\ds\bigcup_{p\in\PP}{\rjpi(G)}$, $\De_1\lhd\De_2\lhd\cdots\lhd\De_N$.

For each $i=1,\ldots,N$, let $p_{(i,1)},\dots,p_{(i,k_i)}$ denote all primes $p$ for which  $\De_i\in \rjpi(G)$.
Observe that
$$
 \sum_{p\in\PP}  \abs{\rjpi(G)}=k_1+\cdots+k_{N}.
$$
For each $i=1,\ldots,N$ and each $j=1,\ldots,k_i$ let us denote
\begin{align}
e_{(i,j)}:=   &  \abs{\set{i'\mid i'\le i \text{ and $\De_{i'}\in \operatorname{RJ}_{p_{(i,j)}}^\infty$}}}, \nonumber \\
H_{(i,j)}:=   &      \De_i+p_{(i,j)}^{e_{(i,j)}}G                                                      \nonumber     \\
n_{(i,j)}:=   &           \prod_{1\le j'<j}p_{(i,j')}^{e_{(i,j')}}                                     \nonumber     \\
n_i:=         &      \prod_{j=1}^{k_i}p_{(i,j)}^{e_{(i,j)}}                                            \nonumber     \\ \nonumber
\end{align}
We want to show that the sequence of $H_{(i,j)}$ lexicographically ordered satisfies Distributivity and infinity of Proposition \ref{prop: cota inferior burden en grups, cas finit}.
\begin{claim}\label{cl: formules càlcul + i interseccions grups Denta+ mG}
\begin{enumerate}
\item Let $\De$, $\De_1\lhd\De_2\lhd\cdots\lhd\De_t$  denote convex subgroups and $m,n,\tira mt$ denote integers. Then
\begin{enumerate}
  \item $nG+mG=(n,m)G$
  \item $\ds \bigcap_{i=1}^t\De+m_iG=\De+[\tira mt]G$
  \item $ \bigcap_{i=1}^t\De_i+m_iG=$  \\
$  \ds= \De_1+m_1\De_2+[m_1,m_2]\De_3+\cdots+[\tira m{{t-1}}]\De_t+[\tira mt]G.$
\end{enumerate}
\item Let $1\le r\le N$ and $1\le s \le k_r$ be given. Then
\begin{enumerate}
  \item  $\ds \bigcap_{(i,j)<(r,s)}H_{(i,j)}=$ \\
  $\De_1+n_1\De_2+[n_1,n_2]\De_3+\cdots +[\tira n{{r-1}}]\De_r+[\tira n{{r-1}},n_{(r,s)}]G.$
  \item $$\ds \bigcap_{(i,j)<(r,s)}\!\!\! H_{(i,j)}+H_{(r,s)}=\De_r+p_{(r,s)}^{e_{(r,s)}-1}G.$$
  \item If $\De_r$ is not the maximum in $\operatorname{RJ}^\infty_{p_{(r,s)}}$, then
  $$
  K_{(r,s)}=\ds \bigcap_{(i,j)\ne(r,s)}\!\!\! (H_{(i,j)}+H_{(r,s)})=\De_r+p_{(r,s)}^{e_{(r,s)}-1}\De_{(r,s)}^{+}+p_{(r,s)}^{e_{(r,s)}}G.
  $$
  Here $\De_{(r,s)}^{+}$ denotes the successor of $\De_r$ in $\operatorname{RJ}^\infty_{p_{(r,s)}}$.
  \item If $\De_r$ is the maximum in $\operatorname{RJ}^\infty_{p_{(r,s)}}$, then
  $$
  K_{(r,s)}=\ds \bigcap_{(i,j)\ne(r,s)}\!\!\! (H_{(i,j)}+H_{(r,s)})=\De_r+p_{(r,s)}^{e_{(r,s)}-1}G.
  $$
  \item If $\De_r$ is not the maximum in $\operatorname{RJ}^\infty_{p_{(r,s)}}$, then
  \newline
  \null\hspace{-20pt}$
      K_{(r,s)}/H_{(r,s)} \simeq \De_{(r,s)}^{+}\Big/ \De_r+p_{(r,s)}\De_{(r,s)}^{+}
      \simeq (\De_{(r,s)}^{+}/ \De_r)\Big/p_{(r,s)}(\De_{(r,s)}^{+}/ \De_r).
  $ \vspace{5pt}
  Here $\De_{(r,s)}^{+}$ denotes the successor of $\De_r$ in $\operatorname{RJ}^\infty_{p_{(r,s)}}$.
  \item If $\De_r$ is the maximum in $\operatorname{RJ}^\infty_{p_{(r,s)}}$, then
  $$
      K_{(r,s)}/H_{(r,s)} \simeq G\Big/ \De_r+p_{(r,s)}G
  \simeq (G/ \De_r)\Big /p_{(r,s)}(G/ \De_r).
  $$
\end{enumerate}
\end{enumerate}
\end{claim}
\begin{proof}
\emph{1.(a).}  Using a B\'{e}zout identity.

\emph{1.(b).} By induction, it is enough to prove $(\De+nG)\cap (\De+mG)=\De+[n,m]G$. The other inclusion being obvious, it suffices to prove $\subseteq$. Assume $a=\de_1+mg_1=\de_2+ng_2$ with $\de_i\in \De$ and $g_i\in G$. Denote $d=(m,n)$, $m=dm'$, $n=dn'$. Let $1=\la m'+\mu n'$ be a B\'{e}zout identity for $n',m'$. Then $a=\la m'a+\mu n'a=\la m'(\de_2+ng_2)+\mu n'(\de_1+mg_1)=(\la m'\de_2+\mu n'\de_1)+[n,m](\la g_2+\mu g_1)\in \De+[n,m]G.$

\emph{1.c.} Since $\bigcap_{i=1}^t\De_i+m_iG=\bigcap_{1\le j\le i\le t}\De_i+m_jG=\bigcap_{i=1}^t\De_i+[\tira mi]G$ (by 1), we may assume $m_i$ divides $m_{i+1}$. We do the proof by induction on $t$. For the case $t=2$ we must show $(\De_1+m_1G)\cap(\De_2+m_2G)\se \De_1+m_1\De_2+m_2G$, the other inclusion being trivial. Assume $a=\de_1+m_1g_1=\de_2+m_2g_2$ with $\de_i\in \De_i$ and $g_i\in G$. Then $\de_2-\de_1\de_1\in m_1G\cap\De_2=m_1\De_2$ and thus $\de_2-\de_1=m_1\de_2'$ for some $\de_2'\in \De_2$. Now $a=\de_1+m_1\de_1'+m_2g_2$. For the inductive step assume the formula holds for $t$. In order to prove it for $t+1$ it suffices to show
$$
\Bigl(\De_1+m_1\De_2+\cdots+m_{t-1}\De_t+m_tG \Bigr) \cap (\De_{t+1}+m_{t+1}G)=
$$

$$
 \De_1+m_1\De_2+m_2\De_3+\cdots+m_{t}\De_{t+1}+m_{t+1}G
$$
The inclusion $\supseteq$ being obvious, assume $a=\de_1+m_1\de_2+m_2\de_3+\cdots m_{t-1}\de_t+m_tg=\de_{t+1}+m_{t+1}g'$ with $\de_i\in \De_i$ and $g,g'\in G$. Then $\de_{t+1}-(\de_1+m_1\de_2+m_2\de_3+\cdots m_{t-1}\de_t)\in m_tG\cap\De_{t+1}=m_t\De_{t+1}$ and thus $\de_{t+1}-(\de_1+m_1\de_2+m_2\de_3+\cdots m_{t-1}\de_t)= m_t\de'_{t+1}$ for some $\de'_{t+1}\in\De_{t+1}$. Hence $a=\de_1+m_1\de_2+m_2\de_3+\cdots m_{t-1}\de_t+m_t\de'_{t+1}+m_{t+1}g'$.

\emph{2.(a).} By \emph{1.(b)}, $\ds \bigcap_{(i,j)<(r,s)}H_{(i,j)}=\left(\bigcap_{i<r}\De_i+n_iG\right)\cap (\De_r+n_{(r,s)}G).$  It follows then using \emph{1.(c)}.

\emph{2.(b).} By \emph{1.(a)}, $H_{(i,j)}+H_{(r,s)}$ is either $G$ or $\De_r+p_{(r,s)}^eG$ for some $e<e_{(r,s)}$.

\emph{2.(c).}  Assume $\operatorname{RJ}^\infty_{p_{(r,s)}}=\set{\tira{\Lambda}m}$ and $\De_r= \Lambda_{e_{(r,s)}}$. By \emph{1.(a),} \emph{1.(c),} it follows that $\ds \bigcap_{(i,j)\ne(r,s)}\!\!\! H_{(i,j)}+H_{(r,s)}= (\De_r+p^{e_{(r,s)}-1}G) \cap(\Lambda_{e_{(r,s)}+1}+p_{(r,s)}^{e_{(r,s)}}G) =\De_r+p_{(r,s)}^{e_{(r,s)}-1} \Lambda_{e_{(r,s)}+1}+p_{(r,s)}^{e_{(r,s)}}G$.

\emph{2.(d).} By \emph{1.(a),}  $\ds \bigcap_{(i,j)\ne(r,s)}\!\!\! H_{(i,j)}+H_{(r,s)}= (\De_r+p^{e_{(r,s)}-1}G)$.

\emph{2.(e).} The map $\De_r^+\to K_{(r,s)}/H_{(r,s)}$ defined by $x\mapsto +p_{(r,s)}^{e_{(r,s)}-1}x+H_{(r,s)}$ is surjective and its kernel is $\De_r+p_{(r,s)}\De_r^+$.

\emph{2.(f).} The map $G\to K_{(r,s)}/H_{(r,s)}$ defined by $x\mapsto +p_{(r,s)}^{e_{(r,s)}-1}x+H_{(r,s)}$ is surjective and its kernel is $\De_r+p_{(r,s)}G$.

\end{proof}
Now we check that the sequence of the $H_{(i,j)}$ satisfies Distributivity. By \emph{2.(a)} and \emph{1.(a)} and using that $\left([\tira n{{r-1}},n_{(r,s)}],p_{(r,s)}^{e_{(r,s)}}\right)=p_{(r,s)}^{e_{(r,s)}-1}$, it follows that $\ds \left(\bigcap_{(i,j)<(r,s)}\!\!\! H_{(i,j)}\right)+H_{(r,s)}=\De_r+p_{(r,s)}^{e_{(r,s)}-1}G$. Hence, bearing in mind \emph{2.(b)} distributivity holds.

Infinity holds by \emph{2.(e).} and \emph{2.(f).}

So by Proposition \ref{prop: cota inferior burden en grups, cas finit} we have an inp-pattern of depth  $\sum_{p\in\PP}  \abs{\rjpi(G)}=k_1+\cdots+k_{N}$. It remains to show we can add an extra row corresponding to the order.  Given a countable set of solutions (which we may suppose all positive) to all systems corresponding to all paths, by saturation, we can choose an element $b_1$ which is bigger to all those solutions.  By Claims \ref{cl: formules càlcul + i interseccions grups Denta+ mG} \emph{1.(b)(c)} and \ref{cl: cosets Delta + p^kG},  the solution set of the system corresponding to any path are cosets of some group of the form $\De_1+m_1\De_2+[m_1,m_2]\De_3+\cdots+[\tira m{{N-1}}]\De_N+[\tira mN]G$ for some $\tira m{{N}}$ positive integers.   By saturation again, adding an element of $[\tira mN]G$ big enough, we obtain a set of solutions of all paths, each one $>b_1$. Iterating this argument, one sees that we can add an extra row of kind $(x>b_i\mid i\in \om)$ to the inp-pattern provided by Proposition \ref{prop: cota inferior burden en grups, cas finit}.

For the proof of the inequality $\dpr(G)\le 1+\sum_{p\in\PP}\abs{\rjpi}$  we use Proposition \ref{prop: la-VCA implica dprang le la} and the QE given in Theorem \ref{thm: QE}. In fact, we are going to see that $G$ is  $N$-VCA, where $N=1+\sum_{p\in\PP}\abs{\rjpi(G)}$. \Wlog\ we may assume $N$ is finite, since the theory is clearly $\aleph_0$-VCA.

 By Theorem \ref{thm: QE}, any formula can be written as a boolean combination of formulas of the following kind:
\begin{eqnarray}
\label{eq: atoms proof strong 11}
  nx &\le & t(\ov y) \\
\label{eq: atoms proof strong 12}
  nx &\equiv & t(\ov y) \mod \De    \\
\label{eq: atoms proof strong 13}
  nx &\equiv & t(\ov y) \mod \De +p^mG   \\
\nonumber
\end{eqnarray}
where $n\in \Z$, $\De\in\rjp(G),\ p\in\PP$, $m\ge 1$ and $t(\ov y)$ a term.

Since the formulas of kind \eqref{eq: atoms proof strong 11} and \eqref{eq: atoms proof strong 12} define convex subsets, they are  boolean combination of definable initial segments. So we may express any formula as a boolean combination of definable initial segments and formulas of kind \eqref{eq: atoms proof strong 13}.  The definable initial segments constitute the first  directed family.


Let $N_p$ denote the number of infinite $p$-regular jumps, i.e., $N_p=\abs{\rjpi(G)}$.
We can write the chain $\rjp(G)=\set{\tira\De{{u-1}}}$ as follows:
\begin{equation}\label{eq: cadena infinite regular jumps}
  \set{0}=\De_0\lhd\De_1\lhd \cdots\lhd\De_{m_1}\lhd\De_{m_1+1}\lhd \cdots 
  \lhd\De_{m_{N_p}}\lhd\De_{m_{N_p}+1}\lhd \cdots \lhd \De_u=G
\end{equation}
where $\rjpi(G)=\set{\De_{m_1},\De_{m_2},\ldots,\De_{m_{N_p}}}$. Therefore $\dim_p(\De_{m_i+1}/ \De_{m_i})$ is infinite for each $i=1\ldots {{N_p}}$, while $\dim_p(\De_{i+1}/ \De_{i})$ is finite for $i\notin\set{m_1,\ldots m_{N_p}}$.

\begin{claim}\label{cl: eliminable cosets}
If $\De\unlhd \De'$ are convex subgroups of $G$ and $\dim_p(\De'/\De)$ is finite then $[\De'+p^rG:\De+p^rG]$ is also finite.
\end{claim}
\begin{proof}The same arguments as in Claim \ref{cl: p-dim finita implica  la formula .... es NA}. In fact it is equal to $p^{mk}$, where $k=\dim_p(\De'/\De)$.
\end{proof}

Now, by Claim \ref{cl: eliminable cosets} any coset of a group of kind $\De_i+p^rG$  with $0\le i\le m_1$ is a finite union of cosets of $\De_{0}+p^rG$. Also any coset coset of kind $\De_i+p^rG$  with $m_i+1\le i\le m_{i+1}$ is a finite union of cosets of $\De_{m_i+1}+p^rG$. Hence, in the boolean combination provided by QE, we may assume the formulas of kind \eqref{eq: atoms proof strong 13} only occur (at most) for $\Delta\in\set{\De_0,\De_{m_1+1},\ldots,\De_{m_{N_p}+1}}$ .

Moreover, if $\De_{m_{N_p}+1}=G$ this convex subgroup does not occur in the Boolean combination, since  $G\notin \rjp$. If $\De_{m_{N_p}+1}\ne G$ then $[G:\De_{m_{N_p}+1}]$ is finite and, by claim \ref{cl: eliminable cosets} we get that the formula $nx \equiv  t(\ov y) \mod \De_{m_{N_p}+1} +p^mG$ is NA. In any case we may assume the formulas of kind \eqref{eq: atoms proof strong 13} (which are not NA) only occur (at most) for $\Delta\in\set{\De_0,\De_{m_1+1},\ldots,\De_{m_{N_p-1}+1}}$ .


By Claim \ref{cl: cosets Delta + p^kG} the formulas of kind \eqref{eq: atoms proof strong 13} with the same $\De$ and $p$ (and maybe different $n,m$) are again a directed family.

Altogether we have found $1+\sum_pN_p$ directed families (1 for de definable initial segments and another one for each $p$ and $\Delta\in\set{\De_0,\De_{m_1+1},\ldots,\De_{m_{N_p-1}+1}}$)
such that any definable set in a single variable is a Boolean combination of formulas of one of those families and NA-formulas. This shows the theory is $1+\sum_pN_p$-VCA.

\end{proof}

\begin{cor}\label{cor: la-VCA le n sii dprank le n}
Let $G$ be an \oag\ with bounded regular rank. Let $\la$ be cardinal number (finite or infinite). Then \tfae
\begin{enumerate}
  \item $G$ has dp-rank at most $\la$
  \item $G$ is $\la$-VCA
\end{enumerate}
\end{cor}
\begin{proof}
\emph{2} implies \emph{1} is just Proposition \ref{prop: la-VCA implica dprang le la}. For the converse, observe that for the proof of the inequality $\dpr(G)\le 1+\sum_{p\in\PP}\abs{\rjpi}$ in Theorem \ref{thm: formula Dp-rang, cas RR acotat} one shows that $G$ is  $N$-VCA, where $N=1+\sum_{p\in\PP}\abs{\rjpi(G)}=\dpr(G)$.
\end{proof}

We say that an \oag\ has \textbf{finite dimension} iff $\dimp(G)$ is finite for all prime $p$. Observe that by Remark \ref{rem: p-dim finita <=> p-rr finit + RJpinf buit} and Theorem \ref{thm: formula Dp-rang, cas RR acotat}, an \oag\ is dp-minimal iff has finite dimension. This characterization of dp-minimality has been independently obtained in \cite{jahnke_simon_walsberg_2017}.

 By Corollary \ref{cor: la-VCA le n sii dprank le n},  an \oag\ is dp-minimal iff it is VCA-minimal.

\begin{cor}\label{cor: formula Dp-rank d'una suma directa de OAGs}
If $G$ and $H$ are non-trivial, then
$\dpr(G\times H)=\dpr(H)+\dpr(H)-1$
\end{cor}
\begin{proof}
It is easy to check that  $\abs{\rjpi(G\times H)}=\abs{\rjpi(G)}+\abs{\rjpi(H)}$.

Hence $\sum_{p\in\PP}\abs{\rjpi(G\times H)}=\sum_{p\in\PP}\abs{\rjpi(G)}+\sum_{p\in\PP}\abs{\rjpi(H)}$.
\end{proof}

%% file: VC-density.tex
\section{VC-density of \oag s}\label{sec: VC-density}

In this section we show that for any \oag\  $G$,  $vc^G(1)$, the VC-density function of $G$ evaluated at $1$, coincides with the \dpra\ of $G$.

Following \cite{ADHMS:2}, de VC-density function of a complete theory $T$ is the function $vc^T: \N\to \R^{\ge0}\cup{\infty}$ defined by
$$
vc^T(n)=\sup\set{vc^*(\vf(\ov x,\ov y)) \mid \vf(\ov x,\ov y)\text{ is an formula and }\abs{\ov x}=n}
$$
Here $vc^*(\vf(\ov x,\ov y))$ denotes the dual VC-density of the partitioned formula $\vf(\ov x,\ov y)$, as defined in \ref{def: dual VC-density}.
For more details on the VC-density function, see \cite{ADHMS:2}.

In the proof of next proposition we use Facts~\ref{fets: sobre el cardinal de S^varphi(A) i els atoms d'una algebra de boole finita}.


\begin{prop}\label{prop: n-VCA implica 1-vc-density=n}
In a $n$-VCA theory $T$, any formula $\vf(x,\ov y)$ has dual VC-density at most $n$. In other words, $vc^T(1)\le n$.
\end{prop}
\begin{proof}
Let $(\Psi_i(x,\ov y) \mid i< n)$ be a collection of $n$ directed families witnessing the theory is $n$-VCA. By fact \ref{fet: finite boolean options}, for each formula $\vf(x,\ov z)$ there is a finite set $\Theta$ of formulas which are Boolean combinations of formulas in $\bigcup_{i<n}\Psi_i(x,\ov y) $ and NA-formulas, such that each instance of $\vf(x,\ov z)$ is equivalent to an instance of some formula in $\Theta$. Let $\Psi'_i$ denote respectively the set of formulas from $\Psi_i$ occurring in the boolean combinations of formulas in $\Theta$ and let $\Upsilon$ be the set of NA-formulas occurring in the boolean combinations of formulas in $\Theta$. Let $N$ be a common upper bound of the number of different sets each formula in $\Upsilon$ can define. We may assume all formulas in $\Upsilon$,$\Theta$ and $\Psi'_i$ have the same parameter variables, say $\ov u$.

Given a set of $\ov z$-parameters $A$, we can choose a set of  $\ov u$-parameters $B$ of size at most $\abs{A}$ such that each instance of $\vf(x,\ov z)$ with parameters from $A$ is an instance of some formula in $\Theta$ with parameters from $B$. Hence, any definable set $\vf(\C,\ov a)$ with $\ov a\in A$ is a boolean combination of sets of kind $\psi(\C,\ov b)$ where $\psi(x,\ov u)\in \bigcup_{i<n}\Psi'_i\cup\Upsilon$ and $\ov b\in B$.

Now it is not difficult to see that $\abs{S^{\vf}(A)}\le
2^{N^{\abs{\Upsilon}}} \prod_{i<n}(\abs{\Psi'_i}\abs{A}+1)$. This holds because  the Boolean algebra generated by $\Upsilon$-formulas with parameters in $B$  has at most $2^{N^{\abs{\Upsilon}}}  $ atoms becase there are at most $N^{\abs{\Upsilon}}$ such nonequivalent formulas. And the boolean algebra generated by the sets defined by $\Psi'_i$-formulas with parameters from $B$ has at most $\abs{\Psi'_i}\abs{A}+1$ atoms.
\end{proof}

\begin{prop}\label{prop: dprank T le vcdensitat T}For any complete theory $T$,
  $$\dpr(T)\le\vc^T(1)$$
\end{prop}
\begin{proof}
  We show that $\dpr(G)\ge n$ implies $\vc^G(1)\ge n$, where $n$ is a natural number. Let $\gen{\psi_i(x,\ov y)\mid i=1\ldots n}$ and $\gen{ a_{i,j} \mid i=1\dots n , \ j\in\om }$ be an ict-pattern for $x=x$. Fix $M\in\om$ and consider
  $$
 B:=\set{ \ov b_{i,j} \mid  i=1\dots n,\ j=1\dots M }
 $$\vspace{-20pt}
 $$
  \psi(x,\ov y):=\psi_1(x,\ov y) \lor \cdots \lor \psi_n(x,\ov y),
  $$
  where
  $$
   \ov b_{i,j}:= a_{1,M+1}, \ldots  a_{i-1,M+1},    a_{i,j},   a_{i+1,M+1} \ldots  a_{n,M+1}.
  $$
  Observe that
  $$
   \psi(x,\ov b_{i,j})\land \lnot\psi(x,\ov b) \vdash  \psi_i(x,\ov a_{i,j})
   $$
   \vspace{-15pt}
   $$
   \lnot \psi(x,\ov b_{i,j}) \vdash   \lnot \psi_i(x,\ov a_{i,j}),
   $$
   where
  $$
   \ov b:= a_{1,M+1}, \ldots  a_{i,M+1},  \ldots  a_{n,M+1}.
  $$
   Since there are $M^n$ different paths within the first $M$ columns of the ict-patttern, this implies that $\abs{S^{\psi}(B)} \ge M^n$. Keeping in mind that $\abs{B}\le nM$, this entails that $\abs{S^{\psi}(B)} \ge (1/n)^n\abs{B}^n$. As $M$ can be arbitrarily large this implies that $\vc^*(\psi)\ge n$.
\end{proof}
\begin{thm}\label{thm: Dp-rank = Vc-densitat}
Let $G$ be an \oag\ with bounded regular rank. Then
$$
\vc^G(1)=\dpr(G).
$$
\end{thm}
\begin{proof}
By Proposition \ref{prop: dprank T le vcdensitat T}, it is enough to show that $\vc^G(1)\le\dpr(G)$. We may assume $G$ has finite \dpra, say equal to $n$. By Corollary \ref{cor: la-VCA le n sii dprank le n} $G$ is $n$-VCA. By Proposition \ref{prop: n-VCA implica 1-vc-density=n}  $\vc^G(1)\le n$.
\end{proof}

%% file: GurevichSchmitt.tex
\section{Gurevich-Schmitt Quantifier Elimination for \oag s}\label{sec: Gurevich-Schmitt QE for OAG}

The definitions and statements 
of this section are taken from
\cite{Schm82a}, \cite{Schm82b} and \cite{Schm84}. All the proofs may be found there.
We start by giving some definitions.

\begin{defn}Let $G$ be an \oag, $g\in G\setminus\{0\}$ and $n\ge2$.
\begin{itemize}
  \item $G$ is called  \textbf{$n$-regular} if for every convex subgroup
$H\ne\{0\}$ of $G$, $G/H$ is $n$-divisible.
  \item $A(g)=$the largest convex subgroup of $G$ not containing $g$.

  \item $B(g)=$the smallest convex subgroup of $G$ containing $g$.

  \item $C(g)=B(g)/A(g)$.

  \item $A_{n}(g)=$the smallest convex subgroup $C$ of $G$ such that $B(g)/C$ is $n$-
regular.

  \item $B_{n}(g)=$the largest convex subgroup $C$ of $G$ such that $C/A(g)$ is $n$-
regular.

  \item $C_{n}(g)=B_{n}(g)/A_{n}(g)$.

  \item For $g=0$ we define $A_{n}(0)=\emptyset$, $B_{n}(0)=\{0\}$.

  \item $F_{n}(g)=$the largest convex subgroup $C$ of $G$ such that
$C\cap(g+nG)=\emptyset$ 
if $g\notin nG$, $F_{n}(g)=\emptyset$ otherwise.

  \item $\Gamma_{1,n}(g)=\{h\in G\mid F_{n}(h)\subset F_{n}(g)\}$.

  \item $\Gamma_{2,n}(g)=\{h\in G\mid F_{n}(h)\se F_{n}(g)\}$.

\end{itemize}
If $g\notin nG$, $\Gamma_{1,n}(g)$ and $\Gamma_{2,n}(g)$ are shown to be
subgroups of $G$ (Facts \ref{fets: A} ii) below) and we can define:
\begin{itemize}
  \item $\Gamma_{n}(g)=\Gamma_{2,n}(g)/\Gamma_{1,n}(g)$.
\end{itemize}

\end{defn}

The sets $A_{n}(g)$, $B_{n}(g)$, $F_{n}(g)$, $\Gamma_{1,n}(g)$ and $\Gamma_{2,n}(g)$
are shown to be definable 
in the language LOG$=\{0,+,-,\le\}$ by a first-order
formula with the only parameter $g$ (see \cite{Schm82a} and \cite{Gu65}).

\begin{fets}\label{fets: A}  If $g,h\ne0$ then:
\begin{enumerate}
\item $A_{n}(g+h)\se A_{n}(g)\cup A_{n}(h)$ and if $A_{n}(g)\subset A_{n}(h)$ then
$A_{n}(g+h)=A_{n}(h)$.
\item $F_{n}(g+h)\se F_{n}(g)\cup F_{n}(h)$ and if $F_{n}(g)\subset F_{n}(h)$ then
$F_{n}(g+h)=F_{n}(h)$.
\item $F_{n}(g+nh)=F_{n}(g)$. $F_{n}(g)=\emptyset$ iff $g\in nG$.
\item $A_{n}(h)\se A_{n}(g)$ iff $B_{n}(h)\se B_{n}(g)$  iff $A_{n}(h)\subset B_{n}(g)$ iff $h\in B_{n}(g)$ iff $g\notin A_{n}(h)$.
\item $A_{n}(h)\subset A_{n}(g)$ iff $B_{n}(h)\subset B_{n}(g)$ iff $B_{n}(h)\se A_{n}(g)$ iff $g\notin B_{n}(h)$ iff $h\in A_{n}(g)$.
\item $A_{n}(g)=\bigcup\{A_p(g)\mid p\text{ a prime divisor of }n\}$.
\item $F_{n}(g)=\bigcap_{h\in G}A_{n}(g+nh)$.
\end{enumerate}
\end{fets}

The language LSP of {\sl spines} contains as non-logical symbols a binary
relation symbol $\le$  and the following monadic relation symbols:
$A$, $F$, $D$ and $\alpha(p,k,m)$ for all $k,m\in\mathbb N\setminus\{0\}$,
and $p$ prime. The {\sl $n$-spine} of $G$, for $n\ge2$, is defined as the
LSP-structure with universe
$$
\{A_{n}(g)\mid g\in G\}\cup\{F_{n}(g)\mid g\in G\},
$$
and with the following interpretation of the relations:

$$
\begin{array}{rl}
C_1\le C_2        &\text{iff \ }C_1\se C_2,      \\
A(C)              &\text{iff \ $C=A_{n}(g)$ for some }g\in G,      \\
F(C)              &\text{iff \ $C=F_{n}(g)$ for some }g\in G,      \\
D(C)              &\text{iff \ $G/C$ is discrete,}      \\
\alpha(p,k,m)(C)  &\text{iff \ $C=F_{n}(g)$ for some $g\in G\setminus nG$ and}\alpha_{p,k}(\Gamma_{n}(g))\ge m,
\end{array}
$$
where  $\alpha_{p,k}(C)$
denotes the dimension of $\big(
p^{k-1}C[p]/ p^{k}C[p] \big)$ as $\F_p$-vector space if it is finite,
and $\alpha_{p,k}(C)=\infty$ otherwise. Since $n\Gamma_{n}(g)=\{0\}$,
$\Gamma_{n}(g)$ is a direct sum of finite cyclic groups of order dividing $n$(see \cite{Kap}) and
$\alpha_{p,k}(\Gamma_{n}(g))$ is the number of cyclic groups of order $p^k$
in this decomposition. Thus  $\alpha_{p,k}(\Gamma_{n}(g))=0$ if $p^k \nmid n$
and the $\alpha(p,k,m)$ are irrelevant for $p^k \nmid n$.

We will denote this structure by $\spn(G)$.

\begin{rem}\label{rem: translation of formulas}
The structure $\spn(G)$ is interpretable in $G$ for every $n\ge
2$. In particular, given any LSP-formula $\psi(z_1 \dots, z_r,
t_1, \dots ,t_s)$, there is a formula in the language of \oags\
$\vf(x_1,\dots, x_r, y_1,\dots, y_s)$ such that for every \oag\ $G$,
and $g_1,\dots, g_r, h_1,\dots, h_s\in G$,
$$
\spn(G)\models \psi(A_{n}(g_1), \dots ,A_{n}(g_r), F_{n}(h_1), \dots, F_{n}(h_s))
$$\vskip-18pt
$$\text{ iff } \ \ \ \ \ \ \   \ \ \ \ \ \ \  G\models \vf(g_1,\dots, g_r, h_1,\dots, h_s) \ \ \ \ \ \ \  \ \ \ \ \ \ \
$$
\end{rem}

Let LOG$^*$ be the definitional expansion of LOG by the following unary
predicates: $M_k$,
{\color{red}
\footnote{In \cite{Schm82a} and \cite{Schm82b} it is used the notation $M(n,k)$. We eliminate $n$ since $M(n,k)$ does not depend on $n$: if $C(g)$ is discrete then $A_{n}(g)=A(g)$, conversely if $C_{n}(g)$ is discrete and $\ov g=k\ov e$ in $C_{n}(g)$, then $A_{n}(g)=A(g)$},
}
$E_{(n,k)}$ and $D_{(p,r,i)}$ for all $n\ge2$, $r\ge1$, $0<i<r$, $k>0$
and $p$ prime.

For $g\ne0$ they are defined by:
\begin{itemize}
  \item
$M_k(g)$
iff $C_2(g)$ is discrete with $1_{C_2(g)}$ denoting its first positive element and $\cla{g}=k\cdot 1_{C_2(g)}$ in $C_2(g)$.

  \item $E_{(n,k)}(g)$ \ iff \ there exists $h\in G$ such that  $F_{n}(g)=A_{n}(h)$, $M(1)(h)$
holds and $\cla{g}=k\cla{h}$ in $\Gamma_{n}(g)$
iff there exists $h\in G$ such that
$F_{n}(g)=A_{n}(h)$, $M(1)(h)$ holds and $F_{n}(g-kh)\subset F_{n}(g)$.
  \item $D_{(p,r,i)}(g)$ \ iff \ $g\in p^rG$ or $\cla{g}\in p^i\Gamma_{p^r}(g)$
iff $g\in p^rG$ or there exists $h\in G$ such that $F_{p^r}(g-p^ih)\subset F_{p^r}(p^ih)=F_{p^r}(g)$
iff $F_{p^r}(p^{r-i}g)\subset F_{p^r}(g)$.
\end{itemize}

\begin{thm}\label{thm: B} For every {\rm LOG}-formula $\varphi(\ov x)$
there exist $n\ge2$,
a quantifier free {\rm LOG}$^*$-formula $\psi_1(\ov x)$, an LSP-formula
$\psi_0(y_1,\dots ,y_m,z_1,\dots,z_r)$, {\rm LOG}-terms $t_i(\ov x)$ for
$i=1,\dots,m$ and $s_i(\ov x)$ for $i=1,\dots,r$ such that for every \oag\ $G$ and every $\ov g\in G^\omega$
$$
G\models\varphi(\ov g)\text{ iff }
\begin{cases}
G\models\psi_1(\ov g)                                     &      \\
\spn(G)\models\psi_0(C_1,\dots,C_m,D_1,\dots,D_r),        &   \end{cases}
$$
where $C_i=A_{n}(t_i(\ov g))$ and $D_i=F_{n}(s_i(\ov g))$.
\end{thm}